\documentclass[11pt]{amsart}
\usepackage{amsmath,amsthm, amscd, amssymb, amsfonts, mathrsfs,mathtools,pb-diagram,pb-xy}
\usepackage{amsxtra, color}
\usepackage{enumitem}
\usepackage[all]{xy}
\usepackage{hyperref}
\usepackage[T2A,T1]{fontenc}
\usepackage[ansinew]{inputenc}
\usepackage{graphicx,fancyhdr}

\usepackage{tikz-cd}

\usepackage{epstopdf}
\usepackage{color}
\usepackage{enumitem}
\usepackage{algorithmic}

\ifpdf
  \DeclareGraphicsExtensions{.eps,.pdf,.png,.jpg}
\else
  \DeclareGraphicsExtensions{.eps}
\fi

\usepackage{amsxtra, color}
\usepackage{amsopn}

\newcommand{\ad}{\operatorname{ad}}
\newcommand{\car}{\operatorname{char}}
\newcommand{\co}{\operatorname{co}}
\newcommand{\comod}[1]{\operatorname{comod} {#1}}
\newcommand{\GL}[1]{\operatorname{GL}\left(#1\right)}
\newcommand{\gr}[1]{\operatorname{gr} {#1}}
\newcommand{\id}[1]{\operatorname{id}_{#1}}
\newcommand{\modu}[1]{\operatorname{mod} {#1}}
\newcommand{\re}{\operatorname{re}}

\newcommand{\vect}[1]{\operatorname{vec}_{#1}}

\newcommand{\sw}[1]{_{(#1)}}

\newcommand{\bq}{\mathfrak{q}}

\newcommand{\cB}{\mathcal{B}}
\newcommand{\cC}{\mathcal{C}}
\newcommand{\cE}{\mathcal{E}}
\newcommand{\cF}{\mathcal{F}}
\newcommand{\cG}{\mathcal{G}}
\newcommand{\cH}{\mathcal{H}}
\newcommand{\cI}{\mathcal{I}}

\newcommand{\cM}{\mathcal{M}}
\newcommand{\cO}{\mathcal{O}}
\newcommand{\cP}{\mathcal{P}}
\newcommand{\cR}{\mathcal{R}}
\newcommand{\cS}{\mathcal{S}}

\newcommand{\cW}{\mathcal{W}}
\newcommand{\cX}{\mathcal{X}}

\newcommand{\ude}{\underline{\Delta}}

\newcommand{\I}{\mathbb{I}}
\newcommand{\N}{\mathbb{N}}
\newcommand{\Z}{\mathbb{Z}}

\newcommand{\HV}[1]{\mathcal{HV}_{#1}}
\newcommand{\yd}[1]{{}^{#1}_{#1}\mathcal{YD}}

\newcommand{\gt}{\mathtt{g}}
\newcommand{\bla}{\boldsymbol{\lambda}}
\newcommand{\bLa}{\boldsymbol{\Lambda}}

\numberwithin{equation}{section}
\theoremstyle{plain}

\newtheorem{theorem}{Theorem}[section]

\newtheorem{definition-theorem}[theorem]{Definition-Theorem}

\newtheorem{question}[theorem]{Question}

\theoremstyle{definition}

\newtheorem{example}[theorem]{Example}

\newtheorem{remark}[theorem]{Remark}
\newtheorem*{remark*}{Remark}

\newtheorem{conj}[theorem]{Conjecture}

\begin{document}

\newcommand\relatedversion{}

\title{Pointed Hopf algebras revisited, with a view from tensor categories}
\address{\noindent FaMAF-CIEM (CONICET) \\ 
	Universidad Nacional de C\'ordoba \\
	Medina Allende s/n, Ciudad Universitaria \\
	5000 C\'ordoba \\
	Rep\'ublica Argentina
}
\author[Iv\'an Angiono]{Iv\'an Angiono}
\email{ivan.angiono@unc.edu.ar}

\thanks{The work was partially supported by CONICET and Secyt (UNC)}

\maketitle

\begin{center}
{\small To Josefina and Antonela.}
	
\vspace{1cm}
\end{center}







\begin{abstract} 
Hopf algebras appear in connection with various problems in Pure Mathematics and Theoretical Physics, mainly through their categories of representations, which are examples of tensor categories. In recent years, there have been major advances in the classification of finite-dimensional Hopf algebras over the complex numbers, especially when restricted to the pointed case. In this survey, we aim to review the main results in this direction, stating the classification theorems recently obtained and the problems still open, and describing the tools needed to solve the problem by means of the so-called Lifting Method of Andruskiewitsch and Schneider. 

We will highlight a more categorical point of view related to these classification results, which offers for a better description of the so-called liftings up to involving cocycle deformations, and which in turn allows certain problems to be reduced to the associated graded pointed Hopf algebras. We will emphasise this point of view with applications to other contexts related to the aforementioned pointed Hopf algebras, such as finite pointed tensor categories, Hopf algebras that do not satisfy the Chevalley property and their finite-dimensional Nichols algebras, and finally module categories over the categories of representations of pointed Hopf algebras.
\end{abstract}

\section{Introduction.}

A Hopf algebra $H$ is an associative algebra with unit which at the same time admits a structure of coalgebra, compatible with the algebra structure, and an antipode $\cS:H\to H$ (an endomorphism of antialgebras, satisfying appropriate identities). The coalgebra structure provides a structure of \emph{monoidal category} on $\modu{H}$, the category of finite-dimensional left $H$-modules, and the antipode makes it \emph{rigid}: that is, the dual vector space is also a representation, whose action is given using $\cS$. 
Dually, the category $\comod{H}$ of finite-dimensional $H$-comodules has the same properties. See e.g. \cite{Radford-book}.

In addition, both categories are $\Bbbk$-linear and abelian. Putting altogether, both $\modu{H}$ and $\comod{H}$ are tensor categories. Moreover, the forgetfut functors to $\vect{\Bbbk}$ (the category of finite-dimensional vector spaces) are a tensor functor. Reciprocally, any tensor category $\cC$ with a fiber functor to $\vect{\Bbbk}$ can be reconstructed as $\cC\simeq\comod{H}$ for $H$ a Hopf algebra, and when the category is finite, we can also identify $\cC$ with the category of $H^*$-comodules \cite{EGNO}. And although not all tensor categories come from Hopf algebras (as far as they cannot admit a fiber functor), they bring prominent examples. As well, many applications of Hopf algebras come trought tensor categories. Thus, providing new examples of Hopf algebras is a key step towards the applications and problems on Hopf algebras themselves and those on tensor categories.

\emph{From now on, we fix an algebraically closed field $\Bbbk$ of characteristic zero.} Classical results establish that commutative Hopf algebras are related with affine group schemes by means of the algebras of functions while cocommutative Hopf algebras come from groups acting on Lie algebras, due to a celebrated result by Cartier-Kostant-Milnor-Moore. Prominent examples of non commutative and non cocommutative Hopf algebras are given by Drinfeld and Jimbo quantum groups \cite{Drinfeld,Jimbo}. Later, Lusztig considered integral forms close to these quantum groups, divided powers versions and finite-dimensional Frobenius-Lusztig kernels, in connection with his studies of representation theory of algebraic groups in positive characteristic \cite{Lusztig-book}.

All these versions of quantum groups are \emph{pointed} Hopf algebras. We recall that a coalgebra is called pointed if any simple subcoalgebra is one-dimensional, and is the dual notion of a basic algebra. Looking at a Hopf algebra $H$, it means that the coradical $H_0$ (the sum of all simple subcoalgebras) becomes the group algebra of the set $G(H)$ of group-like elements once we take into account the product. At the level of tensor categories, it means that $\comod{H}$ is pointed: every simple object $X$ is invertible. Since the late nineties the classification of pointed Hopf algebras and the exploration of connections with other topics were intensively considered. The objective of this paper is to recall the main points about the classification, stating those subfamilies already classified as well as the open questions and connections with other problems. Indeed, Section \ref{sec:classif-pointed-Hopf} contains the relevant definitions and a description of the aforementioned Lifting Method, with an emphasis on describing Nichols algebras. In addition, it reviews the solutions obtained on the classification of finite-dimensional Nichols algebras over finite groups, which is the first step of the Method, solutions to the generation in degree one problem, and finally the results already obtained on liftings together with their relation to cocycle deformations, which led to the strongest results on the classification of pointed Hopf algebras.

In Section \ref{sec:application-classif} we will describe three problems related to pointed Hopf algebras, mainly using the mentioned categorical point of view. First, the natural generalisation in the context of tensor categories, which is the classification of finite pointed tensor categories. We will see that, in the case of abelian groups, there have been significant advances, and that in some sense, the classification reduces to problems of Hopf algebras, or even more so, their associated Nichols algebras. Next, we will describe examples that fall outside the scope of the previous LIfting Method, corresponding to Hopf algebras generated by the coradical, which properly contain it: We will use the previous classification to give basic Hopf algebras, and also to classify finite-dimensional Nichols algebras over these basic Hopf algebras. Finally, we will provide some tools that may be useful for describing all module categories over the tensor category of representations of a pointed Hopf algebra as before.

\subsection*{Notation} 
We denote by $\N_0$ the set of non-negative numbers. For each $\theta\in\N_0$, $(\alpha_i)_{1\le i\le\theta}$ denotes the canonical basis of $\Z^{\theta}$, i.e. $\alpha_i$ has $i$-entry equal to $1$, and all the other entries are $0$. We denote by $\preceq$ the partial order on $\Z^{\theta}$ given by: $(a_1,\cdots,a_{\theta})\preceq(b_1,\cdots,b_{\theta})$ if and only if $a_i\le b_i$ for all $i$.

All vector spaces, tensor products, etc. are over $\Bbbk$. By algebra we mean an associative $\Bbbk$-algebra with unit.
Given an algebra $A$ and a subspace $S\subset A$, $\Bbbk\langle S\rangle$ denotes the subalgebra of $A$ generated by $S$.
For each group $G$, $\Bbbk G$ denotes the group algebra.

Given a coalgebra $(C,\Delta,\epsilon)$ we use Sweedler notation for the coproduct and any left $C$-comodule $(M,\delta)$: 
\begin{align*}
\Delta(c) &= c\sw{1}\otimes c\sw{2}, \quad c\in C, & \delta(m) &= m\sw{-1}\otimes m\sw{0}\in C\otimes M, \quad m\in M.
\end{align*}

\section{On the classification of finite-dimensional pointed Hopf algebras.}
\label{sec:classif-pointed-Hopf}

Next we recall some invariants related with Hopf algebras needed to introduce the main strategy for the classification of pointed Hopf algebras, the so-called Lifting Method; we refer to \cite{Radford-book} for unexplained notation on this topic. Next, we will present the known solutions to the steps of the Lifting Method, with emphasis on the results on finite-dimensional Nichols algebras over finite groups, and finally we will state different classification results for pointed Hopf algebras, obtained by following the steps of this method, and open questions.

\subsection{The Lifting method.}\label{subsec:lifting-method}

Recall that the coradical $C_0$ of a coalgebra $C$ is the sum of all simple subcoalgebras. Starting on $C_0$ we can define recursively the terms $(C_n)_{n\ge 0}$ of the coradical filtration as follows: 
\begin{align*}
C_n&=C_0\wedge C_{n-1} = \Delta^{-1}(C_0\otimes C+C\otimes C_{n-1}), & &n\ge 1. 
\end{align*}
This is an exhaustive coalgebra filtration, i.e. 
\begin{align*}
\Delta(C_n) & \subseteq \sum_{j=0}^n C_j\otimes C_{n-j}\text{ for all }n\ge 0, && \text{and} & C&=\cup_{n\ge 0} C_n. 
\end{align*}
The associated graded vector space $\gr{C}\coloneqq \oplus_{n\ge 0} C_n/C_{n-1}$, where $C_{-1}=0$, is a coradically graded coalgebra: it means that the coradical filtration of $\gr{C}$ is the filtration associated to the graduation.

We say that $c\in C$ is a group-like element if $\Delta(c)=c\otimes c$, and denote by $G(C)$ the set of all group-like elements of $C$. These elements are in correspondence with one-dimensional subcoalgebras, so $\Bbbk G(C) \subseteq C_0$. In case that $\Bbbk G(C)=C_0$ we say that $C$ is \emph{pointed}.

\smallbreak

When we consider a Hopf algebra $H$, the coradical filtration $(H_n)_{n\ge 0}$ is not always a Hopf algebra filtration. Indeed, it is so if and only if $H_0$ is a Hopf subalgebra. We will come back to the case in which the coradical is not a subalgebra in \S \ref{subsec:application-classif-wo-chevalley}. But for the rest of this subsection \textbf{we assume that $H_0$ is a Hopf subalgebra}: This is the case when $H$ is pointed, since $G(H)$ is a group (with the multiplication of $H$) and $\Bbbk G(H)=H_0$. 

Under this assumption $\gr{H}$ is a coradically graded Hopf algebra, $H_0$ is a cosemisimple Hopf algebra and the inclusion $\iota: H_0\hookrightarrow \gr{H}$ and the projection $\pi:\gr{H}\to H_0$ are Hopf algebra maps such that $\pi\circ\iota=\id{H_0}$. Due to Radford-Majid Theorem \cite{Radford-book}, $\gr{H}$ can be reconstructed as $\gr{H}\simeq R\# H_0$, where $R\coloneqq \gr{H}^{\co \pi}$ is the subalgebra of left coinvariants, which is a coradically graded connected Hopf algebra in the category $\yd{H_0}$ of left $H_0$-Yetter Drinfeld modules with the restriction of the multiplication, a suitable coproduct $\ude:R\to R\otimes R$ in $\yd{H_0}$ and the grading $R=\oplus_{n\ge 0} R^n$, $R^n\coloneqq R\cap \gr{H}^n$. The object $R^1\in\yd{H_0}$ is called the \emph{infinitesimal braiding}, the algebra $R\in\yd{H_0}$ is the \emph{diagram} and the subalgebra $\cB(R^1)\coloneqq\Bbbk\langle R^1\rangle$ generated by $R^1$ is the associated \emph{Nichols algebra}: We will come back to this last object in \S \ref{subsec:Nichols-algs}. If $H$ is finite-dimensional, then $R$, and a fortiori $\cB(R^1)$, are so.

\smallskip

\subsubsection{Yetter-Drinfeld modules.} Before introducing Lifting Method, we recall the definition and properties of the category $\yd{K}$, for $K$ a Hopf algebra with bijective antipode:
\begin{itemize}[leftmargin=*]
	\item The objects $V$ of $\yd{K}$ are left $K$-modules and $K$-comodules, whose action and coaction satisfy that
	\begin{align}\label{eq:YD-compatibility}
	\delta(h\cdot v) &= h\sw{1}v\sw{-1}\cS(h\sw{3}) \otimes h\sw{2}\cdot v\sw{0}, & \text{for every }& h\in K, \, v\in V.
	\end{align}
	The maps are simultaneously $K$-linear and $K$-colinear.
	\item The tensor product of two Yetter-Drinfeld modules is a Yetter-Drinfeld module, as well as the dual, with the usual structure of modules and comodules on the tensor product and the dual. Thus, $\yd{K}$ is a monoidal category, and its subcategory of finite-dimensional objects is a tensor category. Moreover, it is braided, with braiding $c_{V,W}:V\otimes W\to W\otimes V$ given by 
	\begin{align}\label{eq:YD-braiding}
		c_{V,W}(v\otimes w) &= v\sw{-1}\cdot w \otimes v\sw{0} & \text{for every }&v\in V, \, w\in W.
	\end{align}
	This allows to consider Hopf algebras in $\yd{K}$.
	\item If $K$ is semisimple and finite-dimensional, then the category $\yd{K}$ is so, as we assume $\car \Bbbk=0$. In particular, this happens when $K=\Bbbk G$, where $G$ is a finite group.
\end{itemize}

\begin{example}\label{exa:YD-groups}
Let $G$ be a group. A Yetter-Drinfeld module $V\in\yd{\Bbbk G}$ is a $G$-graded vector space $V=\oplus_{g\in G}V_g$, which is simultaneously a $G$-module, and such that $h\cdot V_g=V_{hgh^{-1}}$ for all $g,h\in G$. When $G$ is finite, as said above, the category is semisimple, and simple Yetter-Drinfeld modules are parametrized by pairs $(\cO,\rho)$, where $\cO$ is a conjugacy class of $G$ and $\rho:G^h\to\operatorname{Aut}(U)$ is an irreducible representation of $G^h$, the centralizer of an element $h\in \cO$. The corresponding object $M(\cO,\rho)\in\yd{\Bbbk G}$ is explicitly described as follows. As $G$-module, 
\begin{align*}
M(\cO,\rho) &= \operatorname{Ind}^G_{G^h} \rho =\Bbbk G \otimes_{\Bbbk G^h} U.
\end{align*}
Let $\{h_1=h,h_2\cdots, h_m\}$ be an enumeration of $\cO$, and for each $i=1,\cdots,m$, let $z_i\in G$ be such that $h_i=z_ihz_i^{-1}$, with the convention $z_1=e$.
Hence, $M(\cO,\rho) =\oplus_{i=1}^m \Bbbk z_i\otimes U$ and the action is as follows:
\begin{align*}
g\cdot (z_i \otimes W) &= z_j\oplus \rho(g')(w), & \text{where }&  gz_i=z_jg' , &  \text{for each }& g\in G, \, 1\le i\le m, \, w\in W.
\end{align*}
For the action $\delta: M(\cO,\rho)\to \Bbbk G\otimes M(\cO,\rho)$ we simply have:
\begin{align*}
\delta(z_i\otimes w) &=  h_i \otimes (z_i\otimes w),  &  \text{for each }& 1\le i\le m, \, w\in W.
\end{align*}
\end{example}

\subsubsection{The method.} In accordance with what has already been described, we have collected some invariants such as the coradical $H_0$, the diagram $R$, the infinitesimal braiding $R^1$ and the Nichols algebra $\cB(R^1)$. The Lifting Method, originally defined by Andruskiewitsch and Schneider in \cite{AnSch-liftings}, proposes to address the classification of Hopf algebras with a fixed finite-dimensional coradical $K$ (a cosemisimple Hopf algebra, which is also semisimple as $\car \Bbbk=0$) by reversing the steps we took above. In effect, we seek to solve the following problems:
\begin{enumerate} [leftmargin=*,label=\rm{LM\arabic*)}]
\item\label{item:lifting-method-Nichols} Find all $V\in\yd{K}$ such that $\dim\cB(V)<\infty$.

\item\label{item:lifting-method-gen-deg-one} For each $V\in\yd{K}$ as in the previuos step, determine all coradically graded connected Hopf algebras $R=\oplus_{n\ge 0} R^n$ such that $R^1=V$.

Related with this step there is a famous conjecture made by Andruskiewitsch and Schneider \cite{AnSch-Annals}:

\begin{conj}
Assume that $K=\Bbbk G$, where $G$ is a finite group. If $R=\oplus_{n\ge 0} R^n\in\yd{K}$ is a coradically graded connected Hopf algebra such that $\dim R<\infty$, then $R$ is generated as an algebra by $R^1$; in other words, $R=\cB(R^1)$.
\end{conj}

Notice that $R$ is generated by $R^1$ if and only if $R\# K$ is generated as an algebra by $R^1$ and $K$, where $R^1$, respectively $K$, is identified with $R^1\# 1$, respectively $1\# K$, as a subspace of $R\# K$.

\item\label{item:lifting-method-deformations} For each $R$ as in \ref{item:lifting-method-gen-deg-one}, find all Hopf algebras $H$ such that $\gr{H}\simeq R\# K$.
\end{enumerate}

There is a related problem, which was the key step to solve \ref{item:lifting-method-gen-deg-one} and \ref{item:lifting-method-deformations} for the known cases, as we will discuss later in \S \ref{subsec:classif-abelian} and \ref{subsec:classif-non-abelian}:

\begin{enumerate} [leftmargin=*,label=\rm{LM\arabic*')}]
\setcounter{enumi}{1}
\item\label{item:lifting-method-presentation} For each $V\in\yd{K}$ such that $\dim\cB(V)<\infty$, give a presentation by generators and relations of $\cB(V)$.
\end{enumerate}

In the remaining part of this section we will present the known solutions for these steps when $K=\Bbbk G$, with $G$ finite; that is, we will present the state of art of the classification of finite-dimensional pointed Hopf algebras.

\smallbreak

\subsubsection{Cocycle deformations.} Now we give properties of cocycle deformations of Hopf algebras that would be useful later. Recall that a $2$-cocycle on $H$ is a convolution invertible linear map $\sigma:H\otimes H\to \Bbbk$ such that 
\begin{align*}
\sigma(x,1)&=\sigma(1,x)=\epsilon(x), & 
\sigma(x\sw{1},y\sw{2})\sigma(x\sw{2}y\sw{2},z)&=\sigma(y\sw{1},z\sw{1})\sigma(x,y\sw{2}z\sw{2})
\end{align*}
for all $x,y,z\in H$.
Given a $2$-cocycle $\sigma$, set $\cdot_{\sigma}: H\otimes H\to H$ as the map
\begin{align*}
 x\cdot_{\sigma} y & \coloneqq \sigma(x\sw{1},y\sw{1})x\sw{2}y\sw{2}\sigma^{-1}(x\sw{3},y\sw{3}), & x,y & \in H.
 \end{align*}
This map defines a new associative product on $H$, such that $(H,\cdot_{\sigma}, 1,\Delta,\epsilon,\mathcal{S}_{\sigma})$ is a Hopf algebra, for an appropriate antipode $\mathcal{S}_{\sigma}$.
The importance of this notion lies in the fact that the categories of comodules of two Hopf algebras $H$ and $H'$ are tensor equivalent if and only if there exists a $2$-cocycle $\sigma$ on $H$ such that $H'\simeq H_{\sigma}$.

As the coalgebra structure of $H$ and $H_{\sigma}$ is the same, the coradical filtration coincides. In particular, if $H$ is pointed, then $H_{\sigma}$ is so. In \cite{AAGMV} some mild conditions were given on $\sigma$ in order to show that the associated graded Hopf algebras of $H$ and $H_{\sigma}$ coincide. This points towards the idea of constructing all liftings as in \ref{item:lifting-method-deformations} using $2$-cocycles or, equivalently, the so called Hopf-Galois objects. The same work sets out a strategy for this purpose.

\subsection{Nichols algebras.}\label{subsec:Nichols-algs}

Let $K$ be any Hopf algebra, not necessarily semisimple.
As mentioned before, a \emph{Nichols algebra} associated to a Yetter-Drinfeld module $V\in\yd{K}$ is a graded Hopf algebra $\cB(V)=\oplus_{n\ge 0}\cB^n(V)$ in $\yd{K}$, which is connected, i.e. $\cB^0(V)=\Bbbk 1$, coradically graded and is generated by $\cB^1(V)\simeq V$. These conditions are equivalent to the following
\footnote{An element $x$ in a Hopf algebra $\cB$ is \emph{primitive} if $\ude(x)=x\otimes 1+1\otimes x$, and $\cP(\cB)$ denotes the set of all primitive elements of $\cB$.}
\begin{align}
\cB^0(V)&=\Bbbk 1, & \cB^1(V)&=\cP(\cB(V))\simeq V, & &\text{and }\cB^1(V)\text{ generates }\cB(V)\text{ as algebra.}
\end{align}
One can check that an algebra $\cB\in\yd{K}$ with the properties above is unique. In fact, another way to describe $\cB(V)$ is as a quotient $\cB(V)=T(V)/\cI(V)$, where $\cI(V)$ is the maximal graded Hopf ideal of the tensor algebra $T(V)$ generated by elements of degree $\ge 2$ (i.e. the sum of all Hopf ideals with this property).

In addition, the structure of $\cB(V)$ depends really on the braiding map $c_{V,V}:V\otimes V\to V\otimes V$. Indeed, one can define the Nichols algebra $\cB(V)$ for each braided vector space $(V,c)$. We refer to \cite{A-Nichols} and the references therein for more information about Nichols algebras and the unexplained definitions. 

\medbreak

A solution to \ref{item:lifting-method-Nichols} for a fixed semisimple Hopf algebra $K$ involves the description of the defining ideal $\cI(V)$, or a more appropriate inviariant controlling the dimension of $\cB(V)$. Also, an inclusion $W\hookrightarrow V \in \yd{K}$ induces an inclusion $\cB(W)\hookrightarrow \cB(V)$, due to the universal property of Nichols algebras, so we usually argue recursively to solve the referred step of the Lifting Method. Next we discuss the case $K=\Bbbk G$, for $G$ a finite group.

\subsubsection{The Weyl groupoid.}

Next we introduce useful invariants to determine when a Nichols algebra is finite-dimensional, and in such case, to get the dimension. They are the Weyl groupoid and the associated root system. We refer mainly to \cite{HS-book}, but most of the results included here can be also found in \cite{AHS-american}.

\smallbreak

Fix $\theta\in\N$. Set $\I=\{1,\cdots,\theta\}$ and $\cF^{K}_{\theta}$ as the collection of all direct sums $\oplus_{i\in\I} V_i\in\yd{K}$, where $\dim V_i<\infty$. Given $V=\oplus_{i\in\I} V_i\in\cF^H_{\theta}$, the Nichols algebra $\cB(V)$ is $\N_0^{\theta}$-graded: $\cB(V)=\oplus_{\alpha\in\N_0^{\theta}} \cB(V)_{\alpha}$, where $V_i =\cB(V)_{\alpha_i}$.

\textbf{For the rest of this subsection we assume that $V_j\in\yd{K}$ is simple for all $j\ne i$. }
Let $i\in\I$. We say that $V$ is \emph{$i$-finite} if for all $j\ne i$ there exists $m\ge 1$ such that $(\ad V_i)^m V_j=0$ in $\cB(V)$. If so, then we set:
\begin{align}\label{eq:weyl-gpd-aij-defn}
c_{ij}^V & \coloneqq 
\begin{cases}
2, & j=i, \\ 
-\max \{ n\in\N_0 | (\ad V_i)^n V_j\ne 0\}, &   i\ne j;
\end{cases}
&  
\rho_i(V)_j &\coloneqq \begin{cases} V_i^*, & j=i, \\ 
(\ad V_i)^{-c_{ij}^V} V_j, & i\ne j. \end{cases}
\end{align}
Notice that $\rho_i(V)\in\cF^H_\theta$: it is called the \emph{$i$-reflection} of $V$. Also, 
\begin{align*}
c_{ij}^V = -\max\{n\in\N_0 | \cB(V)_{n\alpha_i+\alpha_j}\ne 0\}.
\end{align*}

\begin{remark}\label{rem:i-reflection-relation}	\cite[Corollary 13.4.3 \& Theorem 13.3.9]{HS-book}
For all $j\ne i$, $\rho_i(V)_j\in\yd{K}$ is simple. In addition, there is a close relation between the Nichols algebras $\cB(V)$ and $\cB(\rho_i(V))$.
\end{remark}

Assume that $V$ is $i$-finite for all $i$. Then we set $C^V\coloneqq (a_{ij}^V)\in\Z^{\theta\times\theta}$, which is a generalized Cartan matrix \cite[Lemma 13.4.4]{HS-book}, and define the symmetries $s_i^V:\Z^{\theta}\to\Z^{\theta}$,
$s_i^V(\alpha_j)=\alpha_j-c_{ij}^V\alpha_i$ for all $j\in\I$. If $C^V$ is decomposable, i.e. there exists a partition $\{1,\cdots,\theta\}=I_1\cup I_2$, where $c_{ij}^V=0$ for all $i\in I_1$ and $j\in I_2$ such that
\begin{align}\label{eq:Nichols-decomposition}
\cB(V) &\simeq \cB(U_1)\otimes \cB(U_2), & \text{where }& U_{\ell} \coloneqq \oplus_{j\in I_{\ell}} V_j, \, \ell=1,2.
\end{align}
We refer to \cite{AGr-Advances} for a proof. Thus, in order to study finite-dimensional examples, it is enough to just consider the case in which $C^V$ is indecomposable. If so we say that $V$, or either $\cB(V)$ as well, is \emph{braided indecomposable}.

\smallbreak

We say that $V$ \emph{admits all reflections} if $V$ is $i$-finite for all $i$, so $\rho_i(V)$ is well defined for all $i$, and recursively
$\rho_{i_k}\cdots \rho_{i_1}(V)$ is $i$-finite for all $i, i_1,\cdots,i_k\in\I$ and all $k\in\N$.

In this case, we can attach to $V$ a semi-Cartan graph $\cG(V)=\cG(\I,\cX,\rho,(A^W)_{W\in\cX})$, where $\cX$ is the subset of all $W=\oplus_{i\in\I}W_i$ obtained by applying $\rho_i$'s to $V$, up to isomorphism, and $A^W$ is defined as above.

Following \cite[Definition 13.6.3]{HS-book}, the \emph{Weyl groupoid} $\cW(V)$ of $V$ is the subgroupoid of $\cX\times \GL{\Z^{\theta}}\times\cX$ generated by $\sigma_i^W\coloneqq (W,s_i^W,\rho_i(W))$, with $W\in\cX$, $i\in\I$. Thus, every element of $\cW(V)$ is of the form
$(W,\omega,W')$, where $W\in\cX$, $\omega=s_{i_1}^Ws_{i_2}\cdots s_{i_k}$ for some $i_j\in\I$ and $W'=\rho_{i_{k}}\cdots \rho_{i_1}(W)$.
Here, $s_{i_1}^Ws_{i_2}\cdots s_{i_k}$ is a short notation for 
$s_{i_1}^Ws_{i_2}^{\rho_{i_1}(W)}\cdots s_{i_k}^{\rho_{i_{k-1}}\cdots \rho_{i_1}(W)}$ 
as the upper indices are univocally determined. The Yetter-Drinfeld modules $W$ and $W'$ are called \emph{Weyl equivalent} as each one of them is obtained from the other by a chain of reflections $\rho_i$.

As every $\omega$ such that $(W,\omega, W')\in\cW(V)$ is a product of reflections $s_i^U$, we can define the length $\ell(\omega)$ as the minimal number of reflections needed for such an expression. We say that an expression $\omega=s_{i_1}^Ws_{i_2}\cdots s_{i_k}$ is \emph{reduced} when $k=\ell(\omega)$.
Now we set
\begin{align}\label{eq:defn-real-roots}
	\varDelta_{\re}^W & \coloneqq \left\{ \omega(\alpha_i) \, | \, i\in\I, \, W'\in\cX, \, (W,\omega,W')\in\cW(V) \right\} \subset \Z^{\theta}.
\end{align}

\begin{theorem}\cite[14.4.14]{HS-book}
If $V$ admits all reflections and $\cG(V)$ is finite (in particular, if $\dim\cB(V))<\infty$), then $\cG(V)$ is a Cartan graph.
In this case, $\varDelta(V)\coloneqq (\varDelta_{\re}^W)_{W\in\cX}$ is a finite root system as in \cite{HY-Coxetergpd,CH-fin-Weylgpd}.
\end{theorem}

As we are focused in the finite-dimensional case, we can apply the theory of finite root systems. By \cite{HY-Coxetergpd,CH-fin-Weylgpd}
there exists a unique root system attached to $\cG(V)$, and we denote it symply by $\varDelta^V$ (we avoid to mention that every root is real, as this is the unique case as for finite Weyl groups). In addition:
\begin{itemize}
	\item $\varDelta^V=\varDelta^V_+\cup \varDelta^V_-$, where $\varDelta^V_{\pm}=\varDelta^V\cap(\pm \N_0)^{\theta}$, and $\varDelta^V_-=-\varDelta^V_+$.
	\item There exists a unique element $\omega_0^V$ of maximal length starting in $V$.
	\item Fix a reduced expression $\omega_0^V=s_{i_1}^Vs_{i_2}\cdots s_{i_M}$ and set $\beta_k\coloneqq s_{i_1}^V\cdots s_{i_{k-1}}(\alpha_{i_k})$, $1\le k\le M$. Then $\beta_k\in\varDelta^V_+$ for all $k$, they are pairwise different, and moreover $\varDelta^V_+=\{\beta_k | 1\le k\le M \}$.
\end{itemize}
Combining the root system with the Yetter-Drinfeld structure, we set
\begin{align*}
	V_{\beta_k} & \coloneqq \rho_{i_{k-1}}\cdots \rho_{i_1} (V)_{i_k}\in\yd{K}, & &1\le k\le M.
\end{align*}
Following the definitions in \cite{HS-book} again, we can see $V_{\beta_k}$ as a subspace of $\cB(V)_{\beta_k}$ for all $k$.

\begin{theorem}\cite[\S 14.4 \& 14.5]{HS-book}\label{thm:Nichols-decomposed-PBW}
If $\dim\cB(V))<\infty$, then for each $k\in\{1,\cdots, M\}$ the subalgebra $\Bbbk\langle V_{\beta_k}\rangle$ is isomorphic as an algebra to $\cB(V_{\beta_k})$. The multiplication gives an isomorphism $\cB(V) \simeq \Bbbk\langle V_{\beta_1}\rangle \otimes \cdots \otimes \Bbbk\langle V_{\beta_M}\rangle$.
\end{theorem}

Thus the Weyl groupoid (and the associated root system) is a key ingredient both to classify finite-dimensional Nichols algebras as well as to describe their structure.

\subsubsection{Nichols algebras over abelian groups.}
\label{subsubsec:classif-abelian}
Assume now that $K=\Bbbk G$, where $G$ is a finite abelian group. According with Example \ref{exa:YD-groups}, each simple Yetter-Drinfeld module is one-dimensional and labelled by a pair $(h,\chi)$, where $h\in G$ and $\chi\in\widehat{G}$: we denote it simply $\Bbbk_h^{\chi}$. 
Thus a Yetter-Drinfeld module $V$ of dimension $\theta$ is, up to isomorphism, a direct sum $V=\oplus_{i=1}^{\theta} \Bbbk_{h_i}^{\chi_i}$ for some $h_i\in G$ and $\chi_i\in\widehat{G}$. Fix non-zero vectors $x_i\in \Bbbk_{h_i}^{\chi_i}$, so $(x_i)_{1\le i\le \theta}$ is a basis of $V$. The action, the coaction and the braiding $c:V\otimes V\to V\otimes V$ are given by
\begin{align*}
g\cdot x_i &= \chi_i(g)x_i, & \delta(x_i) &= h_i\otimes x_i, & c(x_i\otimes x_j)&=\chi_j(h_i) \, x_j \otimes x_i, &g\in G, &\, 1\le i, j \le \theta,
\end{align*}
Thus, the matrix $\bq=(q_{ij})\in(\Bbbk^{\times})^{\theta\times \theta}$, where $q_{ij}=\chi_j(h_i)$, fully determines the braiding, and thence the corresponding Nichols algebra, denoted by $\cB_{\bq}\coloneqq\cB(V)$. A braided vector space given by a matrix $\bq$ as above is called \emph{of diagonal type}, and a Nichols algebra is of diagonal type if its associated braiding is so.

\begin{example}\label{exa:Nichols-rank-one}
When $\dim V=1$, i.e. $V=\Bbbk_g^{\chi}$ for some $g\in\Gamma$ and $\chi\in\widehat{G}$, the Nichols algebra $\cB_{\bq}$ is determined by $q=\chi(g)$ as follows:
\begin{align*}
	\cB_{\bq} \simeq \begin{cases} 
\Bbbk[x] & \text{if }q \text{ is not a root of unity or }q=1,
\\
\Bbbk[x]/(x^N) & \text{if }q \text{ is a primitive root of unity of order }N\ge 2.
\end{cases}
\end{align*}
\end{example}

The previous example corresponds to the positive part of the quantized enveloping algebra $U_q(\mathfrak{sl}_2)$ when $q$ is not a root of unity, and the positive part of the small quantum group $\mathfrak{u}_q(\mathfrak{sl}_2)$ when $q$ is a root of unity. More general, positive parts of quantized enveloping algebras $U_q(\mathfrak{g})$ when $q$ is not a root of unity, and of the small quantum groups $\mathfrak{u}_q(\mathfrak{g})$ associated to semisimple Lie algebras $\mathfrak{g}$ are examples of Nichols algebras of diagonal type.

\smallbreak

A usual procedure to deal with braided vector spaces of diagonal type is to realize the braiding in a \emph{universal framework}, over the group $\Z^{\theta}$: although this group is infinite, the realization provides nice properties about $\cB_{\bq}$. We usually take $h_i=\alpha_i$, and $\chi_j$ as the unique character on $\Z^{\theta}$ such that $\chi_j(\alpha_i)=q_{ij}$ for all $i$. One of these nice properties is that $\cB_{\bq}$ is $\N_0^{\theta}$-graded, with each $x_i$ in degree $\alpha_i$.

Assume that the root system $\varDelta^{\bq}$ is finite: $\varDelta^{\bq}_+=\{\beta_k | 1\le k\le M \}$ as above. Each $V_{\beta}$ is a one-dimensional vector space spanned by a vector $x_{\beta}\in\cB_{\bq}$ of degree $\beta=\sum_{i=1}^{\theta} b_i\alpha_i\in\Z^{\theta}$, and the action of $\Z^{\theta}$ on $x_{\beta}$ is given by $\chi_{\beta}\coloneqq\chi_1^{b_1}\cdots \chi_{\theta}^{b_{\theta}}$.  Given $\alpha,\beta\in\Z^{\theta}$, set $q_{\alpha\beta}\coloneqq \chi_{\beta}(\alpha)$. By Theorem \ref{thm:Nichols-decomposed-PBW}, $\cB_{\bq}$ has a (restricted) PBW basis
\begin{align*}
& x_{\beta_1}^{a_1}\cdots x_{\beta_M}^{a_M}, &
& 0\le a_i<N_{\beta_i}, &
\text{where }& N_{\beta} \coloneqq \operatorname{ord} q_{\beta\beta}.
\end{align*}
Each PBW generator $x_{\beta_k}$ is a non-zero element of $V_{\beta_k}$ and can be constructed using Lusztig isomorphisms of Drinfeld doubles of Nichols algebras, see \cite{Heck-Drinfeld-doubles,HY-RUMA}. Originally, there was a combinatorial way to construct PBW bases of Nichols algebras (or more generally, Hopf quotient of tensor algebras) of diagonal type by means of Lyndon words, see \cite{Kharchenko-book}, and the root system of $\cB_{\bq}$ was defined starting on the degrees of PBW generators \cite{HY-Coxetergpd}.

Coming back to the Weyl groupoids, the Cartan matrix $C^{\bq}$ in \eqref{eq:weyl-gpd-aij-defn} becomes
\begin{align*}
c_{ij}^{\bq} &= -\min \{ n\in\N_0 | (n+1)_{q_{ii}}(1-q_{ii}^nq_{ij}q_{ji})=0 \}, &  &i\ne j.
\end{align*}
Thus, the definition of $C^{\bq}$, and a fortiori the reflections $\rho_i$, depends really on $\widetilde{q}_{ij}\coloneqq q_{ij}q_{ji}$ more than in the individual values of $q_{ij}$ and $q_{ji}$. In this direction, we introduce the Dynkin diagram of $\bq$ as the graph with $\theta$ vertices labelled with $q_{ii}$, and edges between $i\ne j$ when $\widetilde{q}_{ij}\ne 1$, labelled with this scalar. Two braiding matrices $\bq$ and $\bq'$ are called \emph{twist equivalent} if they have the same Dynkin diagram. 

Notice that we have an edge between $i\ne j$ if and only if $c_{ij}^{\bq}\ne 0$. A braiding matrix $\bq$ is called connected if and only if its Dynlin diagram is so. According with the discussion around \eqref{eq:Nichols-decomposition}, it is enough to determine when a Nichols algebra has finite root system for a braided indecomposable $V$, which means that the Dynkin diagram is connected.  Using the combinatorics of Weyl groupoids, Heckenberger was able to do so.

\begin{theorem}\label{thm:diagonal-classif}
\cite{Heck-classif-diagonal}
Assume that the Dynkin diagram of $\bq$ is connected.
A Nichols algebra $\cB_{\bq}$ has finite root system if and only if $\bq$ belongs to one of the following families:
\begin{enumerate}[leftmargin=*,label=\rm{(\Roman*)}]
	\item\label{item:diagonal-classif-Cartan} \emph{Cartan type} with parameter $q$, $q\ne 0,1$.
	\item\label{item:diagonal-classif-super} \emph{super type} with parameter $q$, $q\ne 0,\pm 1$.
	\item\label{item:diagonal-classif-super-modular} one of families of \emph{(super) modular type}.
	\item\label{item:diagonal-classif-ufo} one of the twelve families of \emph{unidentified type}.
\end{enumerate}
In particular, $\dim \cB_{\bq}<\infty$ if and only if $\bq$ belongs to \ref{item:diagonal-classif-Cartan} with $q$ a root of unity of order $N\ge 2$,   \ref{item:diagonal-classif-super} with $q$ a root of unity of order $N>2$, \ref{item:diagonal-classif-super-modular} or \ref{item:diagonal-classif-ufo}.
\end{theorem}

Braidings of Cartan type were introduced by Andruskiewitsch and Schneider, see \cite{AnSch-liftings} for the precise definition. Each braiding matrix $\bq$ of Cartan type in Heckenberger's list is twist equivalent to the braiding matrix of the positive part of a quantum group with finite Cartan matrix.\footnote{Here, the parameter for Heckenberger's list differs from the one in quantum groups via a reparametrization $q\leftrightarrow q^2$.} Analogously, those of super type are twist equivalent to the braiding matrices of the positive part of a quantized enveloping superalgebra $U_q(\mathfrak{g})$, where $\mathfrak{g}$ is a finite dimensional contragredient Lie superalgebra. Each family in \ref{item:diagonal-classif-super-modular}  and \ref{item:diagonal-classif-ufo} corresponds to a Weyl groupoid class, i.e. all the braiding matrices obtained by applying repeatedly reflections.
Those $\bq$ in \ref{item:diagonal-classif-super-modular} have the same Weyl groupoid as a contragredient Lie (super) algebra over a field of characteristic $p>0$ and involves a root of unity of order $p$ describing the braiding matrix, while Weyl groupoids  in \ref{item:diagonal-classif-ufo} are in principle not related to any object in Lie Theory. See \cite{AA-survey-diagonal} for more details. About the presentation of the defining ideal of each Nichols algebra $\cB_{\bq}$:

\begin{theorem}\label{thm:Nichols-diag-defn-rels}
\cite{Ang-Crelle}
Let $\cB_{\bq}$ be a Nichols algebra of diagonal type and finite root system with braiding matrix $\bq$ of size $\theta$. Then $\cB_{\bq}$ is presented by generators $x_i$, $1\le i\le\theta$ and relations:
\begin{enumerate}[leftmargin=*,label=\rm{(\roman*)}]
	\item\label{item:diagonal-rels-PRV} $x_{\beta}^{N_{\beta}}$, for all $\beta\in\cO^{\bq}_+$ such that $N_{\beta}<\infty$;
	\item\label{item:diagonal-rels-qSerre} $(\ad_c x_i)^{1-c_{ij}^{\bq}}x_j$ for all $i\ne j$ such that $q_{ii}^{c_{ij}^{\bq}}=q_{ij}q_{ji}$, $1-c_{ij}<\operatorname{ord} N_{i}$;
	\item\label{item:diagonal-rels-simple-nonCartan} $x_i^{N_i}$ for each $i$ such that $\alpha_i\notin \cO^{\bq}_+$;
	\item\label{item:diagonal-rels-others} A list of relations on 2, 3 and 4 letters, depending on the Dynkin subdiagram of these letters.
\end{enumerate}

\end{theorem}

Here, a root $\beta\in\varDelta^{\bq}$ is \emph{Cartan} if $q_{\alpha\beta}q_{\beta\alpha}\in\{q_{\beta\beta}^n| n\in\N_0\}$ for all $\alpha\in\varDelta^{\bq}$, and $\cO^{\bq}_+$ denotes the set of positive Cartan roots. If $\bq$ is of Cartan type, e.g. close to positive parts of quantum groups, then all roots are Cartan and we need just relations \ref{item:diagonal-rels-PRV} and 
\ref{item:diagonal-rels-qSerre}. A particular case of \ref{item:diagonal-rels-simple-nonCartan} is when $q_{ii}=-1$ and $q_{ij}q_{ji}\ne \pm1$ for some $j\ne i$: this happens for example for degenerate odd roots of quantized enveloping algebras. Another fact that occurs in the super case is the need for relations different from Serre's quantum relations, in up to four letters, as shown in \cite{Y-super-quantum}.

The presentation has a deep consequence for pointed Hopf algebras, more specifically to answer \ref{item:lifting-method-gen-deg-one}:

\begin{theorem}\label{thm:gen-deg-one-diagonal}
\cite{Ang-Crelle}
Let $R$ be a finite-dimensional coradically graded Hopf algebra in $\yd{\Bbbk G}$, where $G$ is an abelian group. Then $R=\cB_{\bq}$, where $\bq$ is the braiding matrix of $R^1$. 

Therefore, if $H$ is a finite-dimensional pointed Hopf algebra whose radical is abelian, then $H$ is generated by group-like and skew primitive elements, or equivalently
$\gr H=\cB_{\bq}\# \Bbbk G(H)$. 
\end{theorem}

\subsubsection{Nichols algebras over non abelian groups.}
\label{subsubsec:classif-non-abelian}

Now assume that $G$ is a non-abelian group. As said in Example \ref{exa:YD-groups}, simple Yetter-Drinfeld modules correspond to pairs $(\cO,\rho)$, where $\cO$ is the conjugacy class of an element $g\in G$, and $\rho$ is an irreducible representation of the centralizer $G^g$ of $g$. 

One might think that the simplest case to determine when a Nichols algebra is finite-dimensional is when the Yetter Drinfeld modules module is simple, and it is the first to be tackled, thinking on a recursive argument on the number of simple summands. This is not what happens here. Determining when a Nichols algebra $\cB(M(\cO,\rho))$ is finite-dimensional is an open question, while the case where there are two or more summands, even still difficult, has been already solved: the role of the Weyl groupoid is fundamental in controlling the dimension of the Nichols algebra, as we will explain in the next paragraphs.

\smallbreak

As the structure of the Nichols algebra depends on the underlying braided vector space, as we already mentioned,  we can see braided vector spaces associated to \emph{racks} and cocycles \cite{AGr-Advances}. We recall that a rack is a pair $(X,\rhd)$, where $X$ is a set and $\rhd:X\times X\to X$ is a self distributive function such that $\varphi_x\coloneqq x\rhd-:X\to X$ is bijective for all $x\in X$. A $2$-cocycle for a rack $(X,\rhd)$ and a finite-dimensional vector space $W$ is a function $\bq:X\times X\to GL(W)$ such that
\begin{align*}
	\bq_{x,y\rhd z}\bq_{y,z} &= \bq_{x\rhd y,x\rhd z}\bq_{x,z} & \text{for all }&x,y,z\in X.
\end{align*}
Given a $2$-cocycle $\bq$, we get an structure of braided vector space on $V=\Bbbk X\otimes W$:
\begin{align*}
	c\left( (x\otimes w) \otimes (x'\otimes w')  \right) &= (x\rhd x'\otimes \bq_{x,x'}(w')) \otimes (x\otimes w), & \text{for all }&x,x'\in X, \, w,w'\in W
\end{align*}
Prominent examples of racks and $2$-cocycles come from a conjugacy class of a fixed group and a representation of the centralizer of a fixed element of this conjugacy class. In the case when the conjugacy class is finite, we obtain a finite-dimensional braided vector space, and this is the right context to deal with simple Yetter-Drinfeld modules over non abelian groups.

It is expected that most of the Nichols algebras over a non-abelian group are infinite-dimensional. In that sense, one says that a group $G$ \emph{collapses} if every Nichols algebra over $G$ is infinite dimensional. An effective way to determine that a group collapses is by looking at appropriate subracks of all conjugacy classes of the group. This was made for many simple groups, involving a detailed knowledge of the structure and in some case computations with GAP, see e.g. \cite{ACG-last,AFGV-sporadic} for groups of Lie type and sporadic groups, respectively.

A different and promising approach appeared recently in \cite{HMV-Nichols-prime-dim,AHV-2024}, where reduction to positive characteristic was explored to determine that most conjugacy classes collapse, for example for groups of odd order.

\smallbreak

Below we describe some examples of simple Yetter-Drinfeld modules that are relevant to what follows. We will use the notation in \cite{A-Nichols}.

\begin{example}\label{ex:simple-YD-dim2}
Let $g\in G$ be such that $g^G=\{g,g\kappa\}$ for some $\kappa \neq e \in G$. Then $G^{g}=G^{g\kappa}=G^{g^{-1}}$ is a subgroup of index two and $\kappa\in Z(G)$, $\kappa^2=e$. 
Fix $g_0\in G$ such that $g_0g=\kappa gg_0$, and $\chi$ a one-dimensional representation of $G^g$ such that $\chi(g)=-1$. Then $M= M(g^G, \chi)$ is such that $\dim M=2$: we may fix a basis $\{x,y\}$ such that $y=g_0\cdot x$ and the coaction is given by $\rho(x)=g\otimes x$, $\rho(y)=g\kappa \otimes y$. The action satisfies
\begin{align*}
h \cdot x &= \begin{cases} \chi(h) x, & h\in G^g, \\
			\chi(g_0^{-1}h) y, & h\notin G^g;
		\end{cases}
		& h \cdot y &= \begin{cases}
			\chi(g_0^{-1}hg_0) y, & h\in G^g, \\
			\chi(hg_0) x, & h\notin G^g;
		\end{cases}
& &k\in G.
\end{align*}
Thus $M$ is of diagonal type with braiding matrix $\left[\begin{smallmatrix} -1 & -\chi(\kappa) \\ -\chi(\kappa) & -1 \end{smallmatrix}\right]$, so $\dim \cB(M)=4$: the defining relations are
\begin{align*}
x^2&=y^2=0, & xy+\chi(\kappa)yx &=0.
\end{align*}
\end{example}

\begin{example}\label{exa:FK}
Let $X=\cO^m_2$ be the conjugacy class of transpositions in $\mathbb{S}_m$, and $\bq$ the following cocycle:
\begin{align*}
\bq (\sigma, \tau) & = \begin{cases} 1, & \sigma(i)<\sigma(j), \\ -1 & \sigma(i)>\sigma(j). \end{cases}, & \text{where }\tau&=(i\, j), \, i<j.
\end{align*}
The braided vector space $V=\Bbbk X$ has basis $(x_{ij})_{1\le i<j\le m}$, and can be realized as Yetter-Drinfeld module over $\Bbbk\mathbb{S}_m$. We can check that the following quadratic relations hold in $\cB(V)$:
\begin{align}\label{eq:defn-rels-FKn-1}
&x_{ij}^2 =0, & &i<j,
\\\label{eq:defn-rels-FKn-2}
&x_{ij}x_{kl}-x_{kl}x_{ij} =0, & &\# \{i,j,k,l\}=4,
\\\label{eq:defn-rels-FKn-3}
&
\begin{aligned}
&x_{jk}x_{ik}-x_{ij}x_{jk}+x_{ik}x_{ij}=0, \\
&x_{ik}x_{jk}-x_{jk}x_{ij}+x_{ij}x_{ik}=0,
\end{aligned} & &i<j<k.
\end{align}
Let $\mathtt{FK}_m$ be the quotient of $T(V)$ by these relations: it is called the $m$-th Fomin Kirillov algebra as it appeared first in \cite{FK}, and projects onto $\cB(V)$. Moreover,
\begin{itemize}[leftmargin=*]
	\item If $3\le m\le 5$, then $\mathtt{FK}_m=\cB(V)$, and has dimension 12, 576, 8294400, respectively.
	\item If $m\ge 6$, then it is not know when $\mathtt{FK}_m=\cB(V)$ nor if $\mathtt{FK}_m$ is finite-dimensional.
\end{itemize}
\end{example}

\smallbreak

Next, we consider the case where the Yetter-Drinfeld module is not simple. The role played by the Weyl groupoid is crucial, since the symmetries of this groupoid impose very strong restrictions on the groups in which we can realize Yetter-Drinfeld modules with a finite root system. Indeed,  if a Weyl groupoid $\cW_V$ is finite (e.g. when $\dim \cB(V)<\infty$), then there exists an object $W$ such that the Cartan matrix $C^W$ is finite, see \cite{HV-ranktwo}. As a consequence of this fact:

\begin{theorem}\label{thm:rank2-non-ab-groups-presentations}
\cite{HV-ranktwo} Let $V_1$ and $V_2$ be two absolutely simple Yetter–Drinfeld modules over a non-abelian group $G$. Assume that $G$ is generated by the support of 
$V\coloneqq V_1\oplus V_2$,  $\dim\cB(V)<\infty$ and $c_{V_1,V_2}c_{V_2,V_1}\ne \id{V_2\otimes V_1}$. Then $G$ is a quotient of one of the following groups:
\begin{itemize}[leftmargin=*]
\item $T$, the group given by generators $\zeta$, $\sigma_1$, $\sigma_2$ and relations
\begin{align*}
\sigma_1\sigma_2\sigma_1&=\sigma_2\sigma_1\sigma_2, & \sigma_1^3 &=\sigma_2^3, & \zeta\sigma_i &=\sigma_i\zeta, \quad i=1,2;
\end{align*}
	\item $\Gamma_n$, $n\in\{2,3,4\}$, the group given by generators $a$, $b$, $\nu$ and relations
\begin{align*}
ba &= \nu ab, & \nu a &= a\nu^{-1}, & \nu b &= b\nu, & \nu^n&=1.
\end{align*}
\end{itemize}
\end{theorem}

This result allowed to classify in the same paper all finite-dimensional Nichols algebras over Yetter-Drinfeld modules $V\coloneqq V_1\oplus V_2$, where $V_1$ and $V_2$ are two absolutely simple Yetter–Drinfeld modules and $G$ is a non-abelian group generated by the support. Improving the techniques provided there, together with the finiteness of one of the matrices in the Weyl groupoid, Heckenberger and Vendramin were able to extend the result to decompositions with more summands. We will introduce some notation from \cite{HV-classif-non-ab} needed to describe such classification. 

Let $\cE_{\theta}^G$ be the family of Yetter-Drinfeld modules $V\in\yd{\Bbbk G}$ decomposed as $V=\oplus_{j=1}^{\theta} V_j$, where the support of $V$ generates $G$, each $V_j$ is simple and $V$ is \emph{braided-indecomposable}: for each $i\in\{1,\cdots,\theta\}$ there exists $j\ne i$ such that $c_{V_i,V_j}c_{V_j,V_i}\ne \id{V_j\otimes V_i}$. 

In \cite{HV-classif-non-ab} the authors introduce a kind of Dynkin diagrams of the form 
\begin{align*}
\xymatrix{\square \ar@{.}[r] & \square \ar@{.}[r] & \square \ar@{.}[r] & \cdots },
\end{align*}
where
\begin{itemize}
\item There are $\theta$ vertices $\square$; each vertex $\square$ is either a point $\overset{-1}{\bullet}$ labelled with $-1$, two dots $\colon$ or three points $\therefore\,$.
\item A vertex $i$ of the shape $\overset{-1}{\bullet}$ means that $\dim V_i=1$ and $c_{V_i,V_i}=-\id{V_i\otimes V_i}$.
\item A vertex $i$ of the shape $\colon$ means that $\dim V_i=2$ and $c_{V_i,V_i}$ is as in Example \ref{ex:simple-YD-dim2}.
\item A vertex $i$ of the shape $\therefore$ means that $\dim V_i=3$ and $c_{V_i,V_i}$ is as in Example \ref{exa:FK}.
\item If $i\ne j$ are such that $c_{V_i,V_j}c_{V_j,V_i}=\id{V_j\otimes V_i}$, then there exists no edge between $i$ and $j$.
\item On the other hand, if $c_{V_i,V_j}c_{V_j,V_i}\ne \id{V_j\otimes V_i}$, then there exists an edge between vertices $i$ and $j$, which is either a labeled line $\xymatrix{\ar@{-}[r]^{q} & }$ between two dots, where the label $q\in\Bbbk^{\times}-\{1\}$ refers to $c_{V_i,V_j}c_{V_j,V_i}$, or a double labeled arrow $\xymatrix{\ar@{=>}[r]^{q} & }$ between a point and a vertex which is not a point, or a double arrow $\xymatrix{\ar@{==>}[r] & }$ between two vertices which are both not a point.
\end{itemize}
About those Yetter-Drinfled modules in the classification for $\theta=2$, we include the precise reference to \cite{HV-ranktwo} where one can find the detailed description of the structure of $V= V_1\oplus V_2$.

\begin{theorem}
Let $V \in\cE_{\theta}^G$ be such that the support of $V$ generates $G$ and $V$ is braided-indecomposable. Then $\dim \cB(V)<\infty$ if and only if $\theta=2$ and $V$ is one of the following: 
\begin{itemize}[leftmargin=*]\renewcommand{\labelitemi}{$\circ$}
\item \cite[Example 1.5]{HV-ranktwo}, realizable over the group $\Gamma_4$, and $\dim \cB(V)=65536$;

\item \cite[Example 1.7]{HV-ranktwo}, realizable over the group $T$, and $\dim \cB(V)= 80621568$;

\item \cite[Example 1.9]{HV-ranktwo}, realizable over the group $\Gamma_3$, and $\dim\cB(V)$ is either $10368$ or $2304$, depending on the value of a representation involved in the definition;
\item \cite[Example 1.10]{HV-ranktwo}, realizable over the group $\Gamma_3$, and $\dim \cB(V)=10368$;
\item \cite[Example 1.11]{HV-ranktwo}, realizable over the group $\Gamma_3$, and $\dim \cB(V)=2304$;
\end{itemize}	
or else $V$ has one of the following diagrams from \cite{HV-classif-non-ab}:
\begin{itemize}[leftmargin=*]\renewcommand{\labelitemi}{$\circ$}
\item $\alpha_{\theta}: \quad \xymatrix{\colon \ar@{--}[r] & \colon \ar@{--}[r] & \quad \mathbf{\cdots}\quad  \ar@{--}[r] & \colon\ar@{--}[r] & \colon}$, with $\theta\ge 2$;

\item $\beta'_3: \quad \xymatrix{\overset{\xi}{\bullet} \ar@{-}[r]^{\xi^{-1}} & \overset{\xi}{\bullet} \ar@{=>}[r]^{\xi^{-1}} & \overset{-1}{\therefore} }$, where $1+\xi+\xi^2=0$;

\item $\beta''_3: \quad \xymatrix{\overset{\xi}{\bullet} \ar@{=>}[r]^{\xi-1} & \overset{-\xi}{\colon} \ar@{==>}[r] & \overset{-1}{\therefore} }$, where $1+\xi+\xi^2=0$;

\item $\gamma_{\theta}: \quad \xymatrix{\colon \ar@{--}[r] & \colon \ar@{--}[r] & \quad \mathbf{\cdots}\quad  \ar@{--}[r] & \colon\ar@{<=}[r]^{-1} & \overset{-1}{\bullet}}$, with $\theta\ge 3$;

\item $\delta_{\theta}: \quad \xymatrix{ & & & \colon\ar@{--}[d] & \\ \colon \ar@{--}[r] & \colon \ar@{--}[r] & \quad \mathbf{\cdots}\quad  \ar@{--}[r] & \colon\ar@{--}[r] & \colon}$, with $\theta\ge 4$;

\item $\varepsilon_{\theta}: \quad \xymatrix{ & & \colon\ar@{--}[d] & & \\ \colon \ar@{--}[r] & \colon \quad \mathbf{\cdots} &  \colon \ar@{--}[r] & \colon\ar@{--}[r] & \colon}$, with $\theta\in\{6,7,8\}$;

\item $\varphi_4: \quad \xymatrix{\overset{-1}{\bullet} \ar@{-}[r]^{-1} & \overset{-1}{\bullet} \ar@{=>}[r]^{-1} & \colon\ar@{--}[r] & \colon}$.
\end{itemize}
\end{theorem}

\begin{remark}
\begin{enumerate}[leftmargin=*]
	\item There are no missing examples from \cite{HV-ranktwo}. In fact, Example 1.2 corresponds to type $\alpha_2$ (and is realizable over $\Gamma_2$), while Examples 1.3 only exists over a field of characteristic 3.
	The same happens with type $\beta_{\theta}$ in \cite{HV-classif-non-ab}, as only exists over fields of characteristic 3.
	\item When $\theta=2$, examples of the same dimension realizable over $\Gamma_3$ are Weyl equivalent and the Weyl groupoid is standard of type $B_2$, while those examples over $\Gamma_4$ and $T$ are just equivalent to themselves.
	\item Yetter-Drinfeld modules with diagrams $\beta_3'$ and $\beta_3''$ are Weyl equivalent. All the remaining diagrams are just Weyl equivalent to themselves and their Weyl groupoids are standard (indeed, the name of the diagram is a \emph{greek} version of the associated finite Cartan matrix).
\end{enumerate}	
\end{remark}
	
There exists a common pattern on Yetter-Drinfeld modules of types $\alpha_{\theta}$, $\gamma_{\theta}$, $\delta_{\theta}$, $\epsilon_{\theta}$ and $\phi_4$. Assume that we have one of these types and the group $G$ is generated by the support. Each simple Yetter-Drinfeld submodule $V_i$ of dimension two is as in Example \ref{ex:simple-YD-dim2}: it has a basis $x_i$, $x_{\overline{i}}$, with coaction given by $g_i$, $\kappa g_i$, respectively. As in loc. cit. $\kappa$ is central and such that $\kappa^2=1$. Then $G$ fits into a central extension $1\to \Z_2\to G\to\Gamma\to 1$, where $\Gamma$ is an abelian group and $\Z_2$ corresponds to $\{1,\kappa\}$. In addition, if $V_j$ is another two-dimensional Yetter-Drinfeld such that $i$ and $j$ are connected, then $V_j$ has a basis $x_j$, $x_{\overline{j}}$, with coaction given by $g_j$, $\kappa g_j$, and $g_ig_j=\kappa g_j g_i$.

We say that a Yetter-Drinfeld module in the above list is \emph{almost diagonal} if it decomposes as a direct sum of Yetter-Drinfeld modules of types $\alpha_{\theta}$, $\gamma_{\theta}$, $\delta_{\theta}$, $\epsilon_{\theta}$ and $\phi_4$. The name relies in the fact described below. To do so, first we need some notation. 
Given a $2$-cocycle $\sigma\in H^2(G,\Bbbk)$, there exists a tensor functor $F_{\sigma}:\yd{\Bbbk G}\to\yd{\Bbbk G}$ which essentially twists the $G$-action, see e.g. \cite[\S 3.1]{AngLSa-standard}.

\begin{theorem}\label{thm:almost-diagonal}
\cite{AngLSa-standard}
Let $G$ be a finite non abelian group, $V\in\yd{\Bbbk G}$ almost diagonal such that its support generates $G$. 
Then there exists $\sigma\in H^2(G,\Bbbk)$ such that $F_{\sigma}(V)$ is of diagonal type.
\end{theorem}

Up to the twist in Theorem \ref{thm:almost-diagonal}, these Yetter-Drinfeld modules come from the folding construction given in \cite{Len}.
Using this fact the authors exhibited a presentation by generators and relations of all Nichols algebras with almost diagonal braidings, using the defining relations of the associated braided vector spaces of diagonal type, which are of Cartan type with label $q=-1$. 

About the remaining Yetter-Drinfeld modules in Heckenberger-Vendramin classification, one of them was fully studied in \cite{AngSa-non-abelian}. It was denoted  $\HV{1}$ and corresponds to \cite[Example 1.9]{HV-ranktwo}, so we know that $\dim\cB(\HV{1})=10368$. The associated Nichols algebra was fully studied there, including a presentation by generators and relations, and the generation in degree one problem. In fact:

\begin{theorem}\label{thm:gen-deg-one-non-abelian}
Let $G$ be a finite group and $R$ a finite-dimensional coradically graded Hopf algebra in $\yd{\Bbbk G}$. 
Assume that $V\coloneqq R^1$ is either as in Example \ref{exa:FK} for $n=3,4$, or $\HV{1}$, or else almost abelian. Then $R=\cB(V)$.

Therefore, if $H$ is a finite-dimensional pointed Hopf algebra whose infinitesimal braiding is one of those listed above, then $H$ is generated by group-like and skew primitive elements, or equivalently $\gr H=\cB(V)\# \Bbbk G(H)$.
\end{theorem}
\begin{proof}
When $V$ is as in Example \ref{exa:FK} for $n=3,4$, we refer to \cite[Theorem 1.2]{GIV-FKalgebras} and the references therein. 
For $\HV{1}$, see \cite[Theorem 4.4]{AngSa-non-abelian}.
Finally, the almost abelian case was treated in \cite[Theorem 4.1]{AngLSa-standard}.
\end{proof}

Thus, after considering the main results of \cite{AngLSa-standard,AngSa-non-abelian}, it remains to describe Nichols algebras and generation in degree one problem for five of the Yetter-Drinfeld modules with $\theta=2$, and those of types $\beta_3'$ and $\beta_3''$.

\subsection{Pointed Hopf algebras over abelian groups.}
\label{subsec:classif-abelian}

Due to Theorem \ref{thm:gen-deg-one-diagonal}, the final step of the classification is to get all deformations of $\cB_{\bq}\# \Bbbk G$, where $G$ is a finite abelian group and $\cB_{\bq}$ is a finite-dimensional Nichols algebra of diagonal type admitting a \emph{realization} over $G$: there exist $g_i\in G$ and $\chi_j\in\widehat{G}$ such that $\chi_j(g_i)=q_{ij}$ for all $i,j$.

\smallskip

A first answer to this problem was given in \cite{AnSch-Annals}, where all liftings were computed under the assuption that the order of $G$ is not divisible by 2, 3, 5, 7. Under this restriction, all possible matrices $\bq$ are of Cartan type with a root of unity $q$ of order not divisible by these primes. In addition, the presentation by generators and relations only involves powers of root vectors and quantum Serre relations from Theorem \ref{thm:Nichols-diag-defn-rels}, and quantum Serre relations are not deformed. Also, $N_{\beta}=N$ is the order of $q$, and relations $x_{\beta}^{N}$ are deformed by a recursive formula involving one scalar $\lambda_{\beta}$ for each root $\beta$, subject to the condition: $\lambda_{\beta}=0$ if either $\chi_{\beta}^{N}\ne \epsilon$ or else $g_{\beta}^N=1$.

Later on it was proved in \cite{Ma-cocycle-def} that all liftings in \cite{AnSch-Annals} are cocycle deformations of the associated coradically graded Hopf algebra.

\smallskip

A key step in solving the general case (without restrictions on the order of the group) was to observe that many relations can eventually be deformed, which makes a recursive deformation process harder to describe explicitly, and that it is necessary to involve the cocycle deformations beforehand, rather than a posteriori as in the aforementioned case. To do so, we fix a set of defining relations $\cR_{\bq}=\{r_1,\cdots, r_\ell\}$, where each $r_j\in T(V)$ has degree $\gamma_j\in\N_0^{\theta}$. We assume that the order of the $\N_0$-degrees is not decreasing; that is, if $\gamma_i\le \gamma_j$, then $i<j$. Let $\eta_j$ be the associated element of $\widehat{G}$ describing the $G$-action on $r_j$, and $\gt_j\in G$ the element describing the coaction: if $\gamma_j=\sum_{i=1}^{\theta} a_i\alpha_i$, then $\eta_j=\prod_{i=1}^{\theta}\chi_i^{a_i}$ and $\gt_j=\prod_{i=1}^{\theta}g_i^{a_i}$. Set
\begin{align}\label{eqn:defn-parameter-liftings-diagonal}
\bLa \coloneqq \{ \bla=(\lambda_j)\in \Bbbk^{\ell} | \lambda_j=0 \text{ if either }\eta_j\ne \epsilon \text{ or else }\gt_j=1 \},
\end{align}
a generalization of the set defined in \cite{AnSch-Annals}. For each $\bla\in\bLa$ we will define a lifting $\cH(\bla)$ of $\cB_{\bq}\# \Bbbk G$ which is a quotient of $T(V)\# \Bbbk G$ by relations obtained recursively by \emph{deforming} the relations $r_j\in T(V) \subset T(V)\# \Bbbk G$:
\begin{itemize}[leftmargin=*]
\item Set $\widetilde{r}_1\coloneqq r_1-\lambda_1(1-\gt_1)$, and $\cH_1(\bla)\coloneqq T(V)\# \Bbbk G/\widetilde{r}_1$.
\item Assume that $\cH_j(\bla)$ was already defined. Then there exists $\widehat{r}_{j+1}\in \cH_j(\bla)$ such that $r_{j+1}-\widehat{r}_{j+1}$ is $(1,\gt_{j+1})$-primitive, essentially unique. Then set \begin{align*}
\widetilde{r}_{j+1}& \coloneqq r_{j+1}-\widehat{r}_{j+1}-\lambda_{j+1}(1-\gt_{j+1}), & \cH_{j+1}(\bla) &\coloneqq \cH_j(\bla)/\widetilde{r}_{j+1}.
\end{align*}
\end{itemize}
It follows the strategy depicted in \cite{AAGMV}. Thus, at the end of this process we get a Hopf algebra $\cH(\bla)\coloneqq \cH_{\ell}(\bla)$, which is a cocycle deformation of $\cB_{\bq}\# \Bbbk G$, and such that $\gr \cH(\bla)\simeq \cB_{\bq}\# \Bbbk G$; in other words we construct a family of liftings of $\cB_{\bq}\# \Bbbk G$. Moreover:

\begin{theorem}\label{thm:liftings-abelian}
	\cite{AngGI-liftings}
Let $H$ be a Hopf algebra. Then $\gr H\simeq \cB_{\bq}\# \Bbbk G$ if and only if there exists $\bla\in\bLa$ such that $H\simeq \cH(\bla)$.
In particular, any finite-dimensional pointed Hopf algebra $H$ with abelian coradical is isomorphic to a cocycle deformation of $\cB_{\bq}\# \Bbbk G(H)$, where $\bq$ gives the infinitesimal braiding of $H$. 
\end{theorem}

In \cite{AngGI-liftings} there also exists a characterization of when different data $\bla$, $\bla'$ give place to isomorphic Hopf algebras $\cH(\bla)\simeq\cH(\bla')$, pointing towards a full classification of finite-dimensional pointed Hopf algebras with abelian coradical. What is missing is an explict description of the $\widehat{r}_j$'s since it might be quite hard and, in a sense, not particularly useful. See \cite{GIJG}, where one of these elements in the Cartan $G_2$ case requires many pages to be written, or \cite{AAGI-CartanA} where some difficulties appear to get a recursive formula in the $A_{\theta}$ case for a root of unity of order 2 or 3.

\subsection{Pointed Hopf algebras over non abelian groups.}
\label{subsec:classif-non-abelian}

Now we assume that $G$ is a finite non-abelian group. As said before, we do not have a full classification of all finite-dimensional Nichols algebras in $\yd{\Bbbk G}$ for general $G$, but still we can go on with the remaining step of the Lifting Method by taking a fixed finite-dimensional Nichols algebra and compute all deformations of $\cB(V)\#\Bbbk G$, as generation in degree one holds by Theorem \ref{thm:gen-deg-one-non-abelian}.

We follow the same strategy, starting with $T(V)\# \Bbbk G$ and quotening by an appropriate set of relations determined recursively by a family of parameters. In order to make this process we need to observe that different relations defining $\cB(V)$ can be obtained by applying the $G$-action to a fixed one. 

\begin{example}
Liftings of Nichols algebras coming from Fomin-Kirillov braidings as in Example \ref{exa:FK} for $n=3,4$ can be found in \cite{GIV-FKalgebras}: we summarize here their constructions. Assume that $V$ has a principal realization over $G$ given by pairs $(g_{ij},\chi_{ij})$ for some $g_{ij}\in G$ and characters $\chi_{ij}$ of $G^{g_{ij}}$.

In $T(V)\# \Bbbk G$ we have that $\sigma x_{ij}^2 \sigma^{-1}=x_{kl}^2$, where $k<l$ is such that $\sigma\cdot x_{ij}=\pm x_{kj}$, and for any pair $i<j$ there exists $\sigma \in G$ such that $\sigma \cdot x_{12}=x_{ij}$. Therefore we need just a  parameter $\lambda_1$ to deform $x_{12}^2=\lambda_1(1-g_{12}^2)$, and from that relations $x_{ij}^2=\lambda_1(1-g_{ij}^2)$ hold in the quotient of $T(V)\# \Bbbk G$ by 
$x_{12}^2-\lambda_1(1-g_{12}^2)$.

Similar situations hold for relations in \eqref{eq:defn-rels-FKn-2} and \eqref{eq:defn-rels-FKn-3}. All in all, the authors proved in that paper that $L$ is a lifting of $\cB(V)\#\Bbbk G$ if and only if $L$ is isomorphic to $\cH(\lambda_1,\lambda_2,\lambda_3)$, the quotient of $T(V)\# \Bbbk G$ by the relations
\begin{align*}
x_{12}^2&=\lambda_1(1-g_{12}^2), 
&
x_{12}x_{34}-x_{34}x_{12}&=\lambda_3(1-g_{12}g_{34}),
\\
& &
x_{23}x_{13}-x_{12}x_{23}+x_{13}x_{12}&=\lambda_2(1-g_{12}g_{13}),
\end{align*}
where the deformation parameters $\lambda_1,\lambda_2,\lambda_3\in\Bbbk$ satisfy the constraints
\begin{itemize}
	\item $\lambda_1=0$ if either $\chi_{12}\ne \epsilon$ or else $g_{12}^2=1$,
	\item $\lambda_2=0$ if either $\chi_{12}\chi_{34}\ne \epsilon$ or else $g_{12}g_{34}=1$,
	\item $\lambda_3=0$ if either $\chi_{12}\chi_{23}\ne \epsilon$ or else $g_{12}g_{23}=1$.
\end{itemize}
\end{example}

Now assume that $G$ is a finite group and $V\in\yd{\Bbbk G}$ has braiding either as in Example \ref{exa:FK} for $n=3,4$, or $\HV{1}$, or else almost abelian with a principal realization $(g_i,\chi_i)$: we have a basis $(x_i)$ of $V$, each $g_{i}\in G$ gives the coaction of $(x_i)$ and each character $\chi_{i}$ of $G^{g_{i}}$ gives the action.
We fix a set of defining relations $\cR_{\bq}=\{r_1,\cdots, r_\ell\}$, where each $r_j\in T(V)$ has degree $\gamma_j\in\N_0^{\theta}$ and does not belong to the $G$-submodule generated by the previous relations. Again we assume that the order of the $\N_0$-degrees is not decreasing. 
Let $\gt_j\in G$ be the element describing the coaction of $r_j$; we can check that the $G$-submodule generated by $r_j$ is of the form $M(G^{\gt_j},\eta_j)$ for $\eta_j$ an element of $\widehat{G^{\gt_j}}$. We set again 
$\bLa \coloneqq \{ \bla=(\lambda_j)\in \Bbbk^{\ell} | \lambda_j=0 \text{ if either }\eta_j\ne \epsilon \text{ or else }\gt_j=1 \}$. 
For each $\bla\in\bLa$ we define a lifting $\cH(\bla)$ of $\cB_{\bq}\# \Bbbk G$ as the quotient of $T(V)\# \Bbbk G$ by relations obtained recursively by \emph{deforming} the relations $r_j\in T(V) \subset T(V)\# \Bbbk G$, such that $\cH(\bla)$ is a cocycle deformation of $\cB_{\bq}\# \Bbbk G$:
\begin{itemize}[leftmargin=*]
	\item $\widetilde{r}_1\coloneqq r_1-\lambda_1(1-\gt_1)$, and $\cH_1(\bla)\coloneqq T(V)\# \Bbbk G/\widetilde{r}_1$.
	\item Once $\cH_j(\bla)$ is defined, set $\widetilde{r}_{j+1}\coloneqq r_{j+1}-\widehat{r}_{j+1}-\lambda_{j+1}(1-\gt_{j+1})$, where $\widehat{r}_{j+1}\in \cH_j(\bla)$ is such that $r_{j+1}-\widehat{r}_{j+1}$ is $(1,\gt_{j+1})$-primitive, and $\cH_{j+1}(\bla)\coloneqq \cH_j(\bla)/\widetilde{r}_{j+1}$.
\end{itemize}
Again we get a Hopf algebra $\cH(\bla)\coloneqq \cH_{\ell}(\bla)$, which is a cocycle deformation of $\cB(V)\# \Bbbk G$, and such that $\gr \cH(\bla)\simeq \cB(V)\# \Bbbk G$. Moreover:

\begin{theorem}\label{thm:liftings-non-abelian}
\cite{GIV-FKalgebras,AngLSa-standard,AngSa-non-abelian}
Let $H$ be a finite-dimensional pointed Hopf algebra. Then $H$ has infinitesimal braiding $V$ as above if and only if there exists $\bla\in\bLa$ such that $H\simeq \cH(\bla)$.

In particular, any finite-dimensional pointed Hopf algebra $H$ with infinitesimal braiding $V$ is isomorphic to a cocycle deformation of the Hopf algebra $\cB(V)\# \Bbbk G(H)$. 
\end{theorem}

\section{Applications to other classification problems.}\label{sec:application-classif}

We would like to describe some other classification problems which are related with the classification of pointed Hopf algebras.

\subsection{Pointed tensor categories.}\label{subsec:application-classif-pointed-cat}

Next we consider a categorical version of our main problem, more precisely the classification of finite pointed tensor categories; we refer to \cite{EGNO} for unexplained notions.

Recall that a tensor category $\cC$ is \emph{pointed} if every simple object is invertible, i.e. has Frobenius-Perron dimension one. Prominent examples of pointed categories are those of comodules over pointed Hopf algebras, but to get all pointed tensor categories we need to consider a generalization of the notion of Hopf algebras $H$. A coquasi-Hopf algebra $H$ is a coalgebra provided with a unit, a multiplication $m:H\otimes H\to H$ which is associative up to an element $\omega:H\otimes H\otimes H\to \Bbbk$, and a antipode, which is a triple $(\cS,\alpha,\beta)$, where $\cS:H\to H$ is a linear map, and $\alpha,\beta\in H^*$ are elements satisfying certain compatibility equations. The category $\comod{H}$ of finite-dimensional $H$-comodules is a tensor category with associativity given by $\omega$, unit given by $\epsilon$ and dual objects determined by the antipode, and admits a quasi-tensor functor to $\vect{\Bbbk}$.

An example of a coquasi-Hopf algebra is given by a finite group $G$ and a normalized 3-cocycle $\omega$ on $G$, i.e. a function $\omega:G\times G\times G\to \Bbbk^{\times}$ such that 
\begin{align*}
\omega(gh,k,\ell)\omega(g,h,k\ell)&= \omega(g,h,k)\omega(g,hk,\ell)\omega(h,k,\ell), & \omega(g,1,h) &=1,
\end{align*}
for all $g,h,k,\ell\in G$. Here $H=\Bbbk^{\omega}G$ has the usual coalgebra and algebra structure but the associativity is given by $\omega$, and in addition 
\begin{align*}
\cS(g)&=g^{-1}, & \alpha(g)&=1, & \beta(g)&=\omega(g,g^{-1},g)^{-1}, & \text{for all  } &g\in G.
\end{align*}
The category $\comod{H}$ is tensor equivalent to the category $\vect{G}^{\omega}$ of $G$-graded vector spaces, where the associativity isomorphism is given by $\omega$.

\smallbreak

Given a pointed coquasi-Hopf algebra $H$, the graded coquasi-Hopf algebra $\gr H=\oplus_{n\ge 0} H^{(n)}$ associated to the coradical filtration is a coradically graded coquasi-Hopf algebra, and $H^{(0)}\simeq \Bbbk^{\omega}G$, where $G$ is the group of group-like elements, and $\omega$ is a normalized 3-cocycle. One may follow an addapted version of the Lifting Method, and try to classify first all finite-dimensional coradically graded pointed coquasi-Hopf algebras, then ask for the corresponding generation in degree one problem, and finally compute all the liftings. Related with the second step, in \cite{EGNO} it was conjectured that any finite pointed tensor category is tensor generated by elements of length $\le 2$, which is a generalization of Andruskiewitsch-Schneider conjecture to the level of tensor categories, and means that every finite-dimensional coradically graded coquasi-Hopf algebra is generated by elements in degree 0 and 1.

The classification of finite pointed tensor categories started in \cite{EG-quasiHopf}, after a couple of works by the same authors dealing with the first steps of the addapted Lifting Method. They solve the case in which the group of simple objects form a cyclic group of prime order $p$. That is, $G=\Z/p\Z$ and the strategy is based on the following facts:
\begin{enumerate}
	\item Take any  finite-dimensional pointed coquasi-Hopf algebra $H$.\footnote{In \cite{EG-quasiHopf} the authors works really with the dual notion, the so-called quasi-Hopf algebras.} Then $H^{(0)}=\Bbbk^{\omega}\Z/p\Z$, and the 3-cocycle $\omega$ is \emph{trivializable} over $\Z/p^2\Z$: it means that the pullback $\pi^*\omega$ under the canonical projection $\pi:\Z/p^2\Z\twoheadrightarrow \Z/p\Z$ is a 3-cocycle over $\Z/p^2\Z$ cohomologous to the trivial cocycle. 
	\item There exists an action on $\modu{H}$ by $\Z/p\Z$ such that the equivariantization of $\modu{H}$ by $\Z/p\Z$ gives a coquasi-Hopf algebra $\widetilde{H}$ with coradical $\Bbbk^{\pi^*\omega}\Z/p^2\Z$. Thus $\widetilde{H}$ is, up to cocycle deformation, isomorphic to a pointed Hopf algebra with coradical $\Bbbk\Z/p^2\Z$.
	\item Using facts from the classification of pointed Hopf algebras \cite{AnSch-liftings}, the authors were able to prove that 
	$\mod{\widetilde{H}}\simeq_{\otimes}\hspace{-8pt} \mod{\widehat{H}}$ for a coradically graded Hopf algebra with coradical $\Bbbk\Z/p^2\Z$.
	\item Using this fact, there exists a positive answer to the conjecture of generation in degree one for finite pointed tensor categories.
	\item Following this path back, the authors calculated all coquasi-Hopf algebras obtained by de-equivariantization of coradically graded Hopf algebras over $\Z/p^2\Z$, thus providing a list of coradically graded pointed co-quasi-Hopf algebras of the form $\cB_{\bq}\#\Bbbk^{\omega}\Z/p\Z$, which classify all finite pointed tensor categories with group of invertible elements $\Z/p\Z$ up to tensor equivalence.
\end{enumerate}

Following the same strategy, the case in which the group is cyclic (of any order coprime with $210$) was developed in \cite{Ang-QuasiHopf}. A key point is that, again, any 3-cocycle over $\Z/n\Z$ is trivializable, thus one relates liftings of coquasi-Hopf algebras with de-equivariantizations of Hopf algebras. In a few not fully precise words, the equivariantization process (coming back from coquasi-Hopf to Hopf algebras) and the cocycle deformation process of Hopf algebras \emph{commute}. Therefore, any pointed tensor category with cyclic group is, up to tensor equivalence, the category of comodules over a coradically graded Hopf algebra obtained by de-equivariantization of a Hopf algebra with coradical $\Z/n^2\Z$.

As a generalization of this idea, the de-equivariantization was used in \cite{AngGa-pointed-tensor} to classify finite-dimensional coradically graded  coquasi-Hopf algebras with coradical $\Bbbk^{\omega}G$, where $G$ is abelian and $\omega$ is trivializable. The final answer relies on the classification of coradically graded pointed tensor algebras. 

Even more, finite-dimensional coradically graded  coquasi-Hopf algebras with abelian coradical were classified in \cite{HLYY-classif-pointed-tensor}. First the authors compute explicitly all 3-cocycles over abelian groups. They essentially proved that if the 3-cocycle is not trivializable, \footnote{In \cite{HLYY-classif-pointed-tensor} they called them abelian, which might be confusing as abelian cocycles refer to other notions on braided tensor categories.} then every Nichols algebra over $\Bbbk^{\omega}G$ is infinite-dimensional. Thus the coradical of any finite-dimensional pointed coquasi-Hopf algebra is $\Bbbk^{\omega}G$ with $\omega$ a trivializable 3-cocycle, and then they construct by hand coquasi-Hopf algebras in this case up to come back to Hopf algebras with abelian coradical: the construction gives the same examples as the de-equivariantization process from \cite{AngGa-pointed-tensor}. 

In the meanwhile, it is proved that any pointed tensor category with abelian group of invertible elements is tensor generated by elements of length $\ge 2$, as conjectured in \cite{EGNO}. The proof is by reduction to the Hopf algebra case, which in turn depends on the presentation by generators and relations of Nichols algebras of diagonal type. Therefore, it is expected that the conjecture is true in a broad sense but one may wonder if there exists a more \emph{categorical} proof, not relying on presentations of Nichols algebras.

\smallbreak

Coming back to the whole classification of pointed tesnor categories, even for abelian groups there is still an open problem:

\begin{question}
Compute the liftings of all pointed coradically graded coquasi-Hopf algebras in \cite{HLYY-classif-pointed-tensor}.
\end{question}

One may wonder if, as in the case of cyclic groups, liftings and de-equivariantizations \emph{commute}, so one refers liftings of coquasi-Hopf algebras to liftings of Hopf algebras, and possibly any lifting is a cocycle deformation of the associated graded coquasi-Hopf algebra, extending what is known for Hopf algebras.

Finally we would like to move to non-abelian groups and see if the program above can be extented to finite pointed tensor categories with non-abelian group of invertible elements.

\subsection{Hopf algebras without Chevalley property.}\label{subsec:application-classif-wo-chevalley}

The Lifting Method depicted in \S \ref{subsec:lifting-method} has an strog assumption on the Hopf algebra structure: The Hopf algebra must satisfy the \emph{Chevalley property}, i.e. the coradical is a Hopf subalgebra. In some cases this assumption does not hold and we need to consider a generalization of the Lifting Method \cite{AC-coFrobenius}. To this end, Andruskiewitsch and Cuadra introduced a new coalgebra filtration $\{H_{[n]}\}_{n\ge 0}$, called the \emph{standard filtration}. Let $H$ be a Hopf algebra with coradical $H_0$. The \emph{Hopf coradical} of $H$ is the subalgebra $H_{[0]}$ generated by $H_0$. The next terms are defined recursively:
\begin{align*}
	H_{[n+1]} & \coloneqq H_{[0]} \wedge H_{[n]} = \Delta^{-1}\left( H_{[0]} \otimes H + H\otimes  H_{[n]} \right), & n&\ge 0.
\end{align*}
As proved in \cite{AC-coFrobenius}, $K=H_{[0]}$ is a Hopf subalgebra with coradical $H_0$, and $\{H_{[n]}\}_{n\ge 0}$ is a Hopf algebra filtration. Also, $K$ is generated as an algebra by the coradical. Starting on these facts, in \cite{AC-coFrobenius} the authors propose a generalized Lifting method associated to the standard filtration. Thus, the starting point is to fix a finite-dimensional Hopf algebra $K$ generated by the coradical and determine all finite-dimensional Nichols algebras in $\yd{K}$. 

\smallskip

As said in \cite[Question I]{AC-coFrobenius}, a first step is to determine all Hopf algebras $K$ generated by a fixed cosemisimple coalgebra $C$. In particular, we may have $K=C$, and $K$ is semisimple, as we have fixed an algebraically closed field $\Bbbk$ of characteristic zero. As the classification of such Hopf algebras is widely open, we cannot attack this question in a broad sense, so in principle we would like to get families of Hopf algebras generated by the coradical.

\begin{example}\label{ex:radford-algebra}
Fix $m,n\ge 2$. The generalized Taft algebra $T_{m,n}$ is the algebra generated by $g,x$ with relations
\begin{align}\label{eq:generalized-taft-relations}
	g^{mn}&=1, & gx &= \xi \, xg, &  x^n &=0, & \xi & \text{ a primitive root of unity of order }n.
\end{align}
It is a Hopf algebra  with coproduct determined by $\Delta(g)=g\otimes g$, $\Delta(x)=x\otimes 1+ g\otimes x$, and of dimension $mn^2$, since the set $\{g^ax^b | 0\le a <mn, \, 0\le b <n\}$ is a basis of $T_{m,n}$. In addition, $T_{m,n}$ is pointed, with $G(T_{m,n})\simeq \Z/ mn \Z$, coradically graded and the infinitesimal braiding is one-dimensional.

The Hopf algebra $T_{m,n}$ has, up to isomorphism, a unique non-trivial lifting $R_{m,n}$, called the Radford algebra: The structure is essentially the same, up to change the last relation in \eqref{eq:generalized-taft-relations} by $x^n=1-g^n$.

The dual Hopf algebra $R_{m,n}^*$ is generated by its coradical, as proved in \cite{ACE-Hopf}.
\end{example}

Following this example, we can look for Hopf algebras generated by the coradical within those already classified. More precisely.

\begin{question}\label{ques:dual-pointed-generated-by-coradical}
Determine for which parameters $\bla\in\bLa$ as in Theorem \ref{thm:liftings-abelian}, respectively \ref{thm:liftings-non-abelian}, the Hopf algebra $\cH(\bla)^*$ is generated by the coradical.
\end{question}

Indeed, Example \ref{ex:radford-algebra} corresponds to the case in which the infinitesimal braiding is one dimensional. Up to finishing the classification of finite-dimensional pointed Hopf algebras, a full answer for the question below will provide the classification of all basic Hopf algebras generated by the coradical.
Notice that the answer will fully depend on the structure as a Hopf algebra.
\smallbreak

On the other hand, an asnwer to the next question of the generalized Lifting Method, that is the determination of all finite dimensional Nichols algebras in $\yd{\cH(\bla)^*}$ for $\bla$ as in Question \ref{ques:dual-pointed-generated-by-coradical}, can be obtained from a categorical point of view. 
As explained in \cite{AA-basic}, when $G$ is abelian, we can take advantage of the following chain of equivalences of braided tensor categories: 
\begin{align}\label{eq:equivalence-YD-liftings}
\cG:\yd{\cH(\bla)^*} & \overset{\sim}{\longrightarrow} \yd{\cH(\bla)}\overset{\sim}{\longrightarrow} \yd{H} \overset{\sim}{\longrightarrow} \modu{D(H)}
\end{align}
where $H=\gr \cH(\bla)$, and $D(H)$ denotes the Drinfeld double of $H$, a certain Hopf algebra with underlying vector space $H\otimes H^*\simeq \cB(V) \otimes \Bbbk (G\times \widehat{G})\otimes \cB(V^*)$. Here we combine the cocycle deformation obtained in Theorem \ref{thm:liftings-abelian} for the second equivalence, together with classical equivalences which hold for any finite-dimensional Hopf algebra. The Hopf algebra $D(H)$ has a triangular decomposition, and any simple module is of the shape $L(\lambda)$, where $\lambda\in\widehat{G\times \widehat{G}} = \widehat{G}\times G$ and $L(\lambda)$ is the simple quotient of the Verma module $M(\lambda)$, the induced module from $D(H)^{\ge}=\cB(V) \otimes \Bbbk (G\times \widehat{G})$. In this case:

\begin{theorem}
\cite[Theorems 1.1 \& 1.2]{AA-basic}
Let $G$ be a finite abelian group, $V\in\yd{\Bbbk G}$ such that $\dim \cB(V)<\infty$, and $\bla\in\bLa$ as in Theorem \ref{thm:liftings-abelian}. Let $\cH(\bla)$ be the associated lifting of $\cB(V)\#\Bbbk G$. 
\begin{enumerate}[leftmargin=*,label=\rm{\roman*)}]
	\item Let $W\in \yd{\cH(\bla)^*}$. If $\dim \cB(W)<\infty$, then $W$ is semisimple.
	\item Let $W\in \yd{\cH(\bla)^*}$ be semisimple, so $\cG(W)=\oplus_{i=1}^n L(\lambda_i)$ for some $\lambda_i\in \widehat{G}\times G$. Then $\dim \cB(W)<\infty$ if and only if $\dim \cB(V\oplus \lambda_1\oplus \cdots \lambda_n)<\infty$, where $\lambda_i\in\yd{\Bbbk G}$ means the corresponding one-dimensional Yetter-Drinfeld module. 
\end{enumerate}
\end{theorem}

This result gives a classification of all finite-dimensional Nichols algebras in $\yd{\cH(\bla)^*}$, which in turn give a classification of all finite-dimensional Nichols algebras of finite-dimensional basic Hopf algebras over abelian groups, see \cite[Theorem 1.3]{AA-basic}. The proof of the second item deeply relies in the splitting technique coming from \cite{AHS-american}: There exists a bosonization inside the braided tensor category $\yd{\Bbbk G}$ inducing a decomposition
\begin{align*}
\cB(V\oplus \lambda_1\oplus \cdots \oplus \lambda_n) &\simeq \cB(L(\lambda_1)\oplus \cdots \oplus L(\lambda_n)) \# \cB(V). 
\end{align*}

A particular case was explicitly computed in \cite{BGGM-dualRadford}, where the authors considered the Radford algebra $R_{m,n}$ from Example \ref{ex:radford-algebra} and described the Yetter Drinfeld modules $W$ as above.

A natural question is if these results can be extended to basic Hopf algebras over non-abelian groups.

\subsection{Module categories over Hopf algebras.}\label{subsec:application-classif-module-cat}

Recall that a $\cC$-module category over a tensor category $\cC$ is a 4-uple $(\cM, \otimes,m,\ell)$, where $\cM$ is an abelian category and $\otimes:\cC\times\cM\to\cM$ is a bifunctor, and 
\begin{align*}
m_{X,Y,M}&:(X\otimes Y)\otimes M \to X\otimes (Y\otimes M), & \ell_M&:\mathbf{1}\otimes M\to M, & X,Y\in\cC, & \, M\in \cM, 
\end{align*}
where $\mathbf{1}\in\cC$ is the unit,
are natural isomorfisms satisfying the following identities:
\begin{align*}
(\id{X}\otimes m_{Y,Z,M})m_{X,Y\otimes Z,M}(a_{X,Y,Z}\otimes\id{M})&= m_{X,Y,Z\otimes M}m_{X\otimes Y,Z,M}, 
\\
(\id{X}\otimes \ell_M)m_{X,\mathbf{1},Y} &= r_X\otimes \id{M},
\end{align*}
for all $X,Y,Z\in\cC$ and $M\in\cM$. Here, $r_X:X\otimes\mathbf{1}\to X$ is the natural isomorfism in $\cC$.
For example, any tensor category $\cC$ is a module category over itself.

\smallskip

Let $\cM$, $\cM'$, $\mathcal{N}$ be $\cC$-module categories. 
\begin{itemize}[leftmargin=*]
\item We say that $\cM$ and $\cM'$ are equivalent if there exist $\cC$-module
functors $F:\cM\to\cM'$ and $G:\cM'\to\cM$ such that the compositions $F\circ G$ and $G\circ F$ are naturally isomorphic to the corresponding identities.

\item The direct sum of these categories is the $\Bbbk$-linear category $\cM\times\cM'$ with the canonical $\cC$-structure. A $\cC$-module category $\mathcal{N}$ is \emph{indecomposable} if $\mathcal{N}$ is not equivalent to a direct sum of two non trivial module categories.

\item $\cM$ is \emph{exact} if it is finite and for every projective object $P\in\cC$ and every $M\in\cM$, $P\otimes M$ is projective in $\cM$.
\end{itemize}
As any exact module category decomposes as the direct sum of finite exact indecomposable modules categories,
it suffices to classify exact indecomposable module categories in order to get all exact module categories.

\smallskip

Let $H$ be a finite-dimensional Hopf algebra: we want to describe exact indecomposable module categories over the tensor category $\cC=\modu{H}$ following \cite{AMom-module-cat-Hopfalg}. Recall that a (left) $H$-comodule algebra $A$ is an associative unital algebra $(A,m,u)$ in the category $\comod{H}$. 
An \emph{$H$-ideal} of $A$ is an ideal of $A$ which is also an $H$-subcomodule. In addition, $A$ is said $H$-simple if the unique $H$-ideals are the trivial ones, $0$ and $A$, and $H$-indecomposable if the unique decomposition $A=I\oplus I'$ with $I, I'$ two $H$-ideals of $A$ is the trivial one: $\{I,I'\}=\{0,A\}$.

Let $A$ be an $H$-comodule algebra. The category $\cM=\modu{K}$ is a $\cC$-module category via the coaction $\lambda:A\to H\otimes A$. 
Explicitly, given an $H$-module $X$ and an $A$-module $M$, we set $X\otimes M\in\cM$ with action
\begin{align*}
	a\cdot (x\otimes m) &= a\sw{-1}\cdot x \otimes a\sw{0}\cdot m, & a\in A, & \, x\in X, \, m\in M.
\end{align*}
This defines a functor $\otimes:\cC\times\cM\to\cM$ making $\cM$ a $\cC$-module category.

Next we recall some results from \cite[\S 1]{AMom-module-cat-Hopfalg} translating properties from module categories to comodule algebras. First, $\cM$ is indecomposable if and only if $A$ is $H$-indecomposable. In addition, any indecomposable exact module category over $\cC$ is of the shape $\modu{A}$ for some $H$-indecomposable comodule algebra $A$. Finally, \cite[Theorem 1.25]{AMom-module-cat-Hopfalg} establishes that indecomposable exact module categories over $\cC$ are classified by $H$-indecomposable comodule algebras up to a natural Morita equivalence at the level of comodule algebras.

\smallskip

As stated in loc.cit., examples of exact indecomposable module categories are provided either when $A$ is an $H$-simple comodule algebra, or when $A$ is a coideal subalgebra. Assume that $H=\cB(V)\#\Bbbk G$. Motivated by \cite{Mom-QLS}, were module categories over quantum linear spaces were classified, we may construct $H$-comodule algebras decomposed as a kind of semidirect product $B\rtimes K$, where $B$ is a comodule algebra of $\cB(V)$ and $K$ is the algebra $\Bbbk_{\psi} F$ of an appropriate subgroup $F\subseteq G$ and a $2$-cocycle $\psi$ of $G$ inducing the multiplication, plus some compatibility data. This leads to interesting problems in Nichols algebras: how to find examples and how to classify coideal subalgebras. 

According with \cite{HS-right-coideal}, if $V$ decomposes as a direct sum of simple objects $V=\oplus_{i=1}^\theta V_i$, the Nichols algebra $\cB(V)$ is $\N_0^{\theta}$-graded as stated previously and all $\N_0^{\theta}$-graded coideal subalgebras of $\cB(V)$ are classified by morphisms of the Weyl groupoid $\cW(V)$ ending in $V$; even more, there exists a partial order on this set and the correspondence gives an isomorphism of partially ordered sets. On morphisms ending in $V$, we say that $w_1\le w_2$ if there exists a reduced expression of $w_2$ such that $w_1$ is the beginning of $w_2$; the partial order $\le$ is called the Duflo order as it generalizes the corresponding notion for Weyl groups, and as in this case, it can be read from the root system. Indeed, if we fix a reduced expression $w=s_{i_1}^V\cdots s_{i_{\ell}}$, the set of positive roots in $\varDelta_V^+$ changing to negative roots up to apply $w^{-1}$ is given by $\{\alpha_{i_1},s_{i_1}^V(\alpha_{i_2}), \cdots, s_{i_1}^V\cdots s_{i_{\ell-1}}(\alpha_{i_{\ell}})\}$, and the order on elemets of $\cW_V$ is given by the order on the subsets of positive roots by means of convex orders on $\varDelta_V^+$, see \cite{Ang-JEMS}.

Thus, combining \cite{HS-right-coideal} and \cite{Mom-QLS} we can give a large family of coideal subalgebras of $H=\cB(V)\#\Bbbk G$, which in turn gives exact indecomposable module categories for $\modu{H}$. 

\smallskip

Invoking again a categorical point of view and Theorem \ref{thm:liftings-abelian} and \ref{thm:liftings-non-abelian}, we can reduce the classification of exact indecomposable module categories, or at least the construction of examples, from general finite-dimensional pointed Hopf algebras to those of the form $H=\cB(V)\#\Bbbk G$, where $G$ and $V$ are as in these theorems. Indeed, given an arbitrary finite-dimensional pointed Hopf algebra $L$ as above,  $\modu{L}$ is categorical Morita equivalent to $\comod{L}$ in the sense of \cite{EGNO}, which in turn coincides with $\comod{H}$ for $H=\gr L$, hence there exists a bijective correspondence between module categories of $\modu{L}$ and $\modu{H}$.

\begin{remark}
\begin{itemize}
	\item A particular instance of the procedure above is when $L=\mathfrak{u}_q{\mathfrak{g}}$, the small quantum group of a semisimple Lie algebra $\mathfrak{g}$ at a root of unity $q$ of odd order $N$: using the previous point of view we reduce the computation to that of $H=\cB(V)\#\Bbbk G$, where $G$ decomposes as the product of $\theta=\operatorname{rk} \mathfrak{g}$ finite cyclic groups and $V=V_1\oplus V_2$ is the direct sum of two braided vector spaces of Cartan type $A$, where $A$ is the finite Cartan matrix of $\mathfrak{g}$, one corresponding to the positive part and another to the negative part of $\mathfrak{u}_q{\mathfrak{g}}$.
	\item Some ideas related with those described above were given in \cite{NSS-module-cat-uqsl2}, where the authors classified all exact indecomposable module categories when $\mathfrak{g}=\mathfrak{sl}_2$, relating this problem with the associated graded object, which is a quantum linear space, and then relating the classification with that in \cite{Mom-QLS}. In \cite{NSS-module-cat-uqsl2} they gave explicit descriptions by generators and relations of all comodule algebras needed for the classification, which seems somewhat unmanageable in the general case and leads to the categorical viewpoint being more appropriate for obtaining  classification results of module categories.
\end{itemize}
\end{remark}

%
%
%
%
%

\section*{Acknowledgments.}

I would like to thank all my co-authors. Each of them influenced not only the collaborative work we did together but also enriched my entire research by providing me different points of view and tools. Without any doubt, this invitation to contribute to the Proceedings of the ICM 2026, in connection with the invited lecture in the Algebra Section, would not have been possible.

\bibliographystyle{siamplain}
\bibliography{example_references}
\end{document}